\documentclass[a4paper,english, 11pt]{article}
\usepackage{amssymb,amsthm,amsmath,amscd,amsfonts,bbm,mathrsfs}
\usepackage{hyperref}
\usepackage[all,cmtip]{xy}

\theoremstyle{theorem}

\newtheorem*{thma}{Theorem A}
\newtheorem*{thmb}{Theorem B}
\newtheorem*{thmc}{Theorem C}
\newtheorem*{thmd}{Theorem D}

\newtheorem{Lemma}{Lemma}
\numberwithin{Lemma}{subsection}
\numberwithin{equation}{subsection}
\newtheorem{Definition}[Lemma]{Definition}
\newtheorem{Theorem}[Lemma]{Theorem}
\newtheorem{Proposition}[Lemma]{Proposition}
\newtheorem{Corollary}[Lemma]{Corollary}

\newtheorem{Convention}[Lemma]{Convention}
\theoremstyle{definition}
\newtheorem{Remark}[Lemma]{Remark}
\newtheorem{Example}[Lemma]{Example}

\newcommand{\disp}{\mathcal{D}\text{isp}}

\newcommand{\Z}{\mathbb{Z}}
\newcommand{\Q}{\mathbb{Q}}

\newcommand{\F}{\mathbb{F}}

\DeclareMathOperator{\supp}{Supp}
\DeclareMathOperator{\codim}{codim}
\DeclareMathOperator{\fzip}{F-zip}

\DeclareMathOperator{\Lie}{Lie}
\DeclareMathOperator{\Ker}{Ker}
\DeclareMathOperator{\Pic}{Pic}

\DeclareMathOperator{\Aut}{Aut}
\DeclareMathOperator{\Hom}{Hom}
\DeclareMathOperator{\Sym}{Sym}

\DeclareMathOperator{\GL}{GL}

\DeclareMathOperator{\Spec}{Spec}

\newcounter{listcounter}

\newskip{\itemsepamount}
\newskip{\topsepamount}
\newenvironment{equivlist}{%
  \begin{list}
    {\upshape (\roman{listcounter})}
    {\setlength{\leftmargin}{23pt}
     \setlength{\rightmargin}{0pt}
     \setlength{\itemindent}{0pt}
     \setlength{\labelsep}{5pt}
     \setlength{\labelwidth}{18pt}
     \setlength{\listparindent}{\parindent}
     \setlength{\parsep}{0pt}
     \setlength{\itemsep}{\itemsepamount}
     \setlength{\topsep}{\topsepamount}
     \usecounter{listcounter}}}
  {\end{list}}

\begin{document}

\title{On the Chow Ring of the Stack of truncated Barsotti-Tate Groups}
  \author{Dennis Brokemper}
  \date{November 2016}
  \maketitle
  
\begin{abstract}
We determine the Chow ring of the stack of truncated displays and more generally the Chow ring of the stack of $G$-zips.
We also investigate the pull-back morphism of the truncated display functor. From this we can determine
the Chow ring of the stack of truncated Barsotti-Tate groups over a field of characteristic $p$ up
to $p$-torsion.
\end{abstract}

\section*{Introduction}
\addcontentsline{toc}{section}{Introduction}
Edidin and Graham (\cite{EG2}) develop an equivariant intersection theory for actions of linear algebraic groups $G$ on algebraic spaces $X$.
For such $G$-spaces they define $G$-equivariant Chow groups $A^G_*(X)$ generalizing Totaros' defintion of the $G$-equivariant Chow ring of a point in \cite{To}. 
They are an invariant of the corresponding quotient stack $[X/G]$, i.e. they are independent of the choice of a presentation. Hence they can be used 
to define the integral Chow group of a quotient stack. 
If $X$ is smooth these groups carry a ring structure making them into commutative graded rings.
Edidin and Graham used their theory to compute the Chow ring
of the stacks $\mathscr{M}_{1,1}$ and $\bar{\mathscr{M}}_{1,1}$ of elliptic curves. In an Appendix to that paper Vistoli computed the Chow ring of 
$\mathscr{M}_2$. Edidin and Fulghesu (\cite{EF}) computed the integral Chow ring of the stack of hyperelliptic curves of even genus. 
In this article we investigate the Chow ring of the stack of truncated Barsotti-Tate groups over a field of characteristic $p>0$. 

Let us denote the stack of level-n Barsotti-Tate groups by $BT_n$.
A level-n BT group has a height and a dimension, which are locally constant functions on the base.
If $BT_n^{h,d}$ denotes the stack of level-n BT groups of constant height $h$ and dimension $d$ we obtain a decomposition
$BT_n = \coprod_{0 \leq d \leq h} BT_n^{h,d}$ into open and closed substacks. For example, if $A$ is an abelian scheme of relative dimension $g$
then its $p^n$-torsion subscheme $A[p^n]$ is a level-n BT group of height $2g$ and dimension $g$. 

Although $BT_n^{h,d}$ has a natural presentation $[X/\GL_{p^{nh}}]$ as a quotient stack
with quasi-affine and smooth $X$ (cf. \cite{We}), it seems unlikely that this presentation can be used directly to compute the Chow ring. 
Instead we relate the stack of truncated Barsotti-Tate groups to a stack whose Chow ring is easier to compute, but still closely related to the Chow ring of $BT_n$. 

Our choice for this stack is the stack $\disp_n$ of truncated displays introduced in \cite{La}. Displays were
first introduced in \cite{Zi} to provide a Dieudonne theory that is valid not only over perfect fields but more generally over $\mathbb{F}_p$-algebras or $p$-adic rings. 
While display are given by an invertible matrix with entries in the ring of Witt vectors $W(R)$, if a basis of the underlying modules is fixed,
a truncated display is given by an invertible matrix over the truncated Witt ring $W_n(R)$. 

Using crystalline Dieudonne theory one can associate to every p-divisible group a display. This induces a morphism $\phi \colon BT \to \disp$ from 
the stack of Barsotti-Tate groups to the stack of displays, which in turn induces a morphism 
$$\phi_n \colon BT_n \to \disp_n.$$ compatible with the truncations on both sides.
By \cite{La} this morphism is a smooth morphism of smooth algebraic stacks over $k$ and an equivalence on geometric points.
\begin{thma}
 The pull-back $\phi_n^* \colon A^*(\disp_n) \to A^*(BT_n)$ is injective and an isomorphism after inverting $p$.
\end{thma}
Let us sketch the proof. Consider a field $L$ and a morphism $\Spec L \to BT_n$. After base change to a finite field extension of $p$-power degree the fiber
$\phi_n^{-1}(\Spec L)$ is equal to the classifying space of an infinitesimal group scheme necessarily of $p$-power degree. 
It follows that the pull-back map of Bloch's higher Chow groups $A_*(\Spec L,m) \to A_*(\phi_n^{-1}(\Spec L),m)$ becomes an isomorphism after inverting $p$.
Using the long localization exact sequence the theorem follows from a limit argument and noetherian induction similar to that in \cite[Proposition 4.1]{Qu}. 
The injectivity assertion follows since $A^*(\disp_n)$ is $p$-torsion free.

Thus to compute the Chow ring of $BT_n$ at least up to $p$-torsion it suffices to compute the Chow ring of $\disp_n$, which is much easier
due to the simpler presentation as a quotient stack. More precisely, if $\disp_n^{h,d}$ denotes the open and closed substack in $\disp_n$ of truncated displays
with constant dimension $d$ and height $h$ we have
$$
\disp_n^{h,d}=[\GL_h(W_n(\cdot))/G_n^{h,d}],
$$
where $W_n$ refers to the ring of truncated Witt vectors and $G_n^{h,d}$ is an extension of $\GL_d \times \GL_{h-d}$ by a unipotent group. 
The following result reduces the calculation of $A^*(\disp_n)$ to the case $n=1$.
\begin{thmb}
 The pull-back $\tau_n^* \colon A^*(\disp_1) \to A^*(\disp_n)$ of the truncation map $\tau_n \colon \disp_n \to \disp_1$ is an isomorphism.
\end{thmb}
This is proved using the factorization 
$$[\GL_h(W_n(\cdot))/G_n^{h,d}] \to [GL_h/G_n^{h,d}] \to [GL_h/G_1^{h,d}]$$
of $\tau_n$ and the fact that the first map is an affine bundle and that $G_n^{h,d}$ is an extension of $G_1^{h,d}$ by a unipotent group. 

In a similar way one shows that the Chow ring of $\disp^{h,d}_1$ coincides with that of the quotienstack
$$
[\GL_h/(\GL_d \times \GL_{h-d})],
$$
where the action is given by conjugation with the Frobenius. This situation is a special case of Proposition \ref{PropChowConj}.

\begin{thmc}
\label{ThdisC}
The following equation holds
 \begin{align*}
A^*(\disp^{h,d}_1) & =A^*_{\GL_d \times \GL_{h-d}}(\GL_h) \\
                   & = \mathbb{Z}[t_1,\ldots,t_h]^{S_d \times S_{h-d}}/((p-1)c_1,\ldots,(p^h-1)c_h), 
\end{align*}
where $c_1,\ldots,c_h$ are the elementary symmetric polynomials in the variables $t_1,\ldots,t_h$.  
\end{thmc}
Moreover, $t_1,\ldots,t_d$ resp.\ $t_{d+1},\ldots,t_h$ are the Chern roots of the vector bundle $\mathcal{L}ie$ resp.\ ${}^t\mathcal{L}ie^{\vee}$ over
$\disp^{h,d}_1$.
Here $\mathcal{L}ie$ is a vector bundle of rank $d$ assigning to a display its Lie algebra and ${}^t\mathcal{L}ie^{\vee}$ is of rank $h-d$ assigning
to a display the dual Lie algebra of its dual display.

It follows that the $\mathbb{Q}$-vectorspace $A^*(\disp^{h,d}_1)_{\mathbb{Q}}$ is finite dimensional of dimension $\binom{h}{d}$, which also equals
the number of isomorphism classes of truncated displays of level $1$ with height $h$ and dimension $d$ over an algebraically closed field. We show that a basis 
is given by the cycles of the closures of the respective EO-Strata. 
We prove this fact in greater generality for the stack of $G$-zips (\cite{PWZ}) in Section 4.4. 
In this section we will also compute the Chow ring of the stack of $G$-zips for a connected algebraic zip datum. As in the case of displays the computation 
can be reduced to the situation of Proposition \ref{PropChowConj}. In fact, truncated displays of level $1$ are a special case of $G$-zips. 

Now by the above results we gain the following information on the Chow ring of the stack of truncated Barsotti-Tate groups.
\begin{thmd}
\begin{equivlist}
\item We have $$A^*(BT^{h,d}_n)_{p}=\mathbb{Z}[p^{-1}][t_1,\ldots,t_h]^{S_d \times S_{h-d}}/((p-1)c_1,\ldots,(p^h -1)c_h),$$ where $c_i$ denotes the $i$-th elementary 
symmetric polynomial in the variables $t_1,\ldots,t_h$ and $t_1,\ldots,t_d$ resp.\ $t_{d+1},\ldots,t_h$ are the Chern roots of $\mathcal{L}ie$ 
resp.\ ${}^t\mathcal{L}ie^{\vee}$.
\item $\dim_{\mathbb{Q}}A^*(BT^{h,d}_n)_{\mathbb{Q}}=\binom{h}{d}$ and a basis is given by the cyclces of the closures of the $EO$-Strata.
\item $$  (\Pic BT^{h,d}_n)_p= \begin{cases} \mathbb{Z}[p^{-1}]/(p-1) & \mbox{if } d=0,h \\
                                          \mathbb{Z}[p^{-1}] \times \mathbb{Z}[p^{-1}]/(p-1) &\mbox{else,}
                            \end{cases}$$
where the generator for the free resp.\ torsion part is $\det(\mathcal{L}ie)$ resp.\ $\det(\mathcal{L}ie \otimes {}^t\mathcal{L}ie^{\vee})$.
\end{equivlist}
\end{thmd}
It would be interesting to know if the Chow ring of $BT_n$ has $p$-torsion, and more specifically if the Picard group of $BT_n$ has $p$-torsion. 
However, since $\phi_n^*$ is injective and the Chow ring of $\disp_n$ is $p$-torsion free, $p$-torsion in the Chow ring of $BT_n$ cannot be constructed using displays. 
\\
\\
{\bf Acknowledgement.} This article is part of the author's PhD thesis. I wish to thank my advisor Eike Lau or his guidance and support, and for many helpful 
comments on earlier versions of this article.
I am also grateful to Jean-Stefan Koskivirta for many valuable discussions on $G$-zips and mathematics in general.
\\
\\
\textit{Terminology and Notation.} 
Every scheme is assumed to be of finite type and separated over the base field $k$. In Section 2 we assume $k$ to be of characteristic $p>0$. 
Algebraic groups are affine smooth group schemes over $k$.
We call an algebraic group $G$ unipotent if $G$ admits a filtration $G=G_0 \supset G_1 \supset \ldots \supset G_e=\{1\}$ by subgroups such that
$G_i$ is normal in $G_{i-1}$ with quotient isomorphic to $\mathbb{G}_a$.
The character group of an algebraic group $G$ will be denoted by $\hat{G}$. 
If $X$ is a scheme $A^*(X)$ will always denote the operational Chow ring of $X$ (\cite[Chapter 17]{Fu}). $A_*(X)$ resp. $CH^*(X)$ will be the Chow group of $X$ 
graded by dimension resp. codimension.
If $X$ is an algebraic space over $k$ with a left action of an algebraic group $G$ we will refer to $X$ as a $G$-space. We write $[X/G]$ for the corresponding quotient stack. 
If $G$ acts freely on $X$, i.e. the stabilizer of every point is trivial, then $[X/G]$ is an algebraic space. In this case we will write $X/G$ instead of $[X/G]$ 
and call $X \to X/G$ the principal bundle quotient of $X$ with structure group $G$.

\tableofcontents

\section{Equivariant Intersection Theory}
\subsection{Equivariant Chow Groups}
Consider an algebraic group $G$ over $k$. By \cite[Lemma 9]{EG2} we can find a representation $V$ of $G$, and an open subset $U$ in $V$ such that the complement 
of $U$ has arbitrary high codimension, and such that the principal bundle quotient $U/G$ exists in the category of schemes. If $X$ is an algebraic space on which
$G$ acts then $G$ acts diagonally on $X \times U$ and we will denote the principal bundle quotient $(X \times U)/G$ by $X_G$. 

\begin{Convention}
\label{ConMixedSpace}
We call a pair $(V,U)$ consisting of a $G$-representation $V$ and an open subset $U$ a good pair for $G$  
if $G$ acts freely on $U$, i.e. the stabilizer of every point is trivial. 
Sometimes we will call the quotient $X_G=(X \times U)/G$ a mixed space for the $G$-space $X$. 
If $(V,U)$ is a good pair for $G$ with $\codim(U^c,V)>i$ we will also call $(X \times U)/G$ an approximation of $[X/G]$ up to codimension $i$. 
\end{Convention}

If $X$ has dimension $n$ the $i$-th equivariant Chow group $A^G_i(X)$ is defined in the following way. Choose a good pair $(V,U)$ for $G$ such that
the complement of $U$ has codimension greater than $n-i$. Then one defines
$$A_i^G(X)=A_{i+l-g} (X_G),$$ where $l$ denotes the dimension of $V$ and $g$ is the dimension of $G$. 
The definition is independent of the choice of the pair $(V,U)$ as long as  $\codim (U^c,V)>n-i$ holds (\cite[Definition-Proposition 1]{EG2}). 

The equiviariant Chow groups have the same functorial properties as ordinary Chow groups (\cite[Section 2]{EG2}). 
In particular, we have an operational equivariant Chow ring $A^*_G(X)$ (\cite[Section 2.6]{EG2}), i.e. an element $c \in A^i_G(X)$ 
consists of operations $c(Y \to X) \colon A_*^G(Y) \to A^G_{*-i}(Y)$ for each $G$-equivariant map $Y \to X$ 
that are compatible with flat pull-back, proper push-forward and Gysin homomorphisms. 

We will denote by $CH^*_G(X)$ the $G$-equivariant Chow group of $X$ graded by codimension.
If $X$ is a pure dimensional $G$-scheme and $(V,U)$ a good pair for $G$ with $\codim(U^c,V)>i$ then 
$$CH^j_G(X)=CH^j((X \times U)/G)$$ 
for all $j \leq i$. This motivates the term ``approximation of $[X/G]$ up to codimension $i$'' in Convention \ref{ConMixedSpace}. 

If $X$ is smooth then $CH^*_G(X)$ carries a ring structure which makes it into a commutative graded ring 
with unit element. Moreover, there is a natural isomorphism $A^*_G(X) \cong CH^*_G(X)$ of graded rings (\cite[Proposition 4]{EG2}). 

By \cite[Proposition 16]{EG2} the equivariant Chow groups do not depend on the presentation as a quotient, meaning
if $X$ is a $G$-space and $Y$ is an $H$-space such that $[X/G] \cong [Y/H]$, then $A^G_{i+g}(X)=A^H_{i+h}(Y)$, where $g=\dim G$ and $h=\dim H$.
Hence one can define the Chow group of a quotient stack $[X/G]$ to be $$A_i([X/G])=A_{i+g}^G(X)$$ with $g=\dim G$.
By \cite[Proposition 19]{EG2} one has $A^*([X/G])\cong A_*([X/G])$, whenever $X$ is smooth.

\subsection{Higher Equivariant Chow Groups}
The reason why we shall need higher Chow groups is that they extend the localization exact sequence to the left.
Higher Chow groups were introduced by Bloch in \cite{Bl}.
For a scheme $X$ higher Chow groups $A_i(X,m)$ are defined as the homology of the complex $z_i(X,*)$, where $z_i(X,m)$ 
is the group of cycles of dimension $m+i$ in $X \times \Delta^m$ meeting all faces properly.
For $m=0$ one gets back the usual Chow group $A_*(X)$ and $A_i(X,m)$ may be non-trivial for $-m \leq i \leq \dim X$.  
The definition of these higher Chow groups also works for algebraic spaces. 

In order to define $G$-equivariant versions $A_*^G(X,m)$ of higher Chow groups we need the homotopy property for the mixed spaces $X_G$, i.e. the pull-back map
$$ A_*(X_G,m) \to A_*( \mathcal{E},m) $$ 
for a vector bundle $\mathcal{E}$ over $X_G$ is an isomorphism.
This is true for any scheme if $\mathcal{E}$ is trivial by \cite[Theorem 2.1]{Bl}. To prove the assertion for arbitrary vector bundles
one needs the localization exact sequence of higher Chow groups proved by Bloch in the case of quasi-projective schemes: 
If $X$ is an equidimensional, quasi-projective scheme over $k$ and $Y \subset X$ a closed 
subscheme with complement $U=X-Y$, then there is a long exact sequence of higher Chow groups
\begin{align*}
\ldots \to A_*(Y,m) \to A_*(X,m)& \to A_*(U,m) \to A_*(Y,m-1) \\
 &\to \ldots \to A_*(Y) \to A_*(X) \to A_*(U) \to 0.
\end{align*}
For a proof see \cite[Lemma 4]{EG2} and \cite[Theorem 3.1]{Bl}.

\begin{Remark}
 Levine extended Blochs proof of the existence of the long localization exact sequence to all separated schemes of finite type over $k$ 
 (\cite[Theorem 1.7]{Le}). 
 Hence for the equivariant higher Chow groups to be well defined it suffices that we can choose the mixed spaces to be separated schemes over $k$.
 However, in all applications we have in mind the conditions of Lemma \ref{LeAdm} will be satisfied.
\end{Remark}

\begin{Lemma}
\label{LeAdm}
Let $G$ be an algebraic group and $X$ a normal, quasi-projective $G$-scheme. Then for any $i>0$ there is a representation $V$ of $G$
and an invariant open subset $U \subset V$ whose complement has codimension greater than $i$ such that $G$ acts freely on $U$ and the principal bundle quotient
$(X \times U)/G$ is a quasi-projective scheme. In other words, the quotient stack $[X/G]$ can be approximated by quasi-projective schemes.
\end{Lemma}

\begin{proof}
Embed $G$ into $\GL_n$ for some $n$. Then there is a representation $V$ of $\GL_n$ and an open subset $U \subset V$, whose complement
has codimension greater than $i$ such that $U/\GL_n$ is a Grassmannian (See \cite[Lemma 9]{EG2}). Since $\GL_n$ is special the $\GL_n/G$-bundle
$\pi \colon U/G \to U/\GL_n$ is locally trivial for the Zariski topology, and we will first show that $\pi$ is quasi-projective. 

Since $\GL_n/G$ is quasi-projective and normal there is an ample $\GL_n$-linearizable line bundle $L \to \GL_n/G$ (\cite[Section 5.7]{Th}). Then 
$$(U \times L)/\GL_n \to (U \times (\GL_n/G))/\GL_n=U/G$$ is a line bundle relatively ample for $\pi$. This shows that $\pi$ is quasi-projective.
The same holds then for $U/G$. 
Again by \cite[Section 5.7]{Th} there is an ample $G$-linearizable line bundle on $X$. The pull-back to $X \times U$ is then relatively ample 
for the projection $X \times U \to U$. Applying \cite[Proposition 7.1]{GIT} to this situation yields the claim.
\end{proof}

\begin{Definition}
\label{DefAdm}
\begin{equivlist}
\item A pair $(V,U)$ will be called an admissible pair for a $G$-scheme $X$ if $(V,U)$ is a good pair for $G$ and if the mixed space $X_G$ 
is quasi-projective and (locally) equidimensional over $k$. $X$ will be called an admissible $G$-scheme if for any $i$ there is an admissible pair 
$(V,U)$ for $X$ with $\codim(U^c,V)>i$.
\item If $X$ is an admissible $G$-scheme we define its higher equivariant Chow groups to be
 $$
 A_i^G(X,m)=A_{i+l-g}(X_G,m),
 $$
 where $g=\dim G$ and $X_G$ is formed from an $l$-dimensional admissible pair $(V,U)$ such that $\codim(U^c,V)>\dim X + m-i$. 
\item We will say that a stack $\mathscr{X}$ admits an admissible presentation if there exists an admissible $G$-scheme $X$ such that $\mathscr{X}=[X/G]$.
\item Let $\mathscr{X}$ be a quotient stack that admits a presentation $\mathscr{X}=[X/G]$ by an admissible $G$-scheme $X$.
We define the higher equivariant Chow groups of $\mathscr{X}$ as 
$$A_*(\mathscr{X},m)=A^G_{*+g}(X,m)$$
where $g=\dim G$.
 \end{equivlist}
\end{Definition}

\begin{Remark}
 The proof that Definition \ref{DefAdm} (ii) resp. (iv) is independent of the choice of the admissible pair $(V,U)$ resp. the presentation $[X/G]$ 
 is the same as for ordinary equivariant Chow groups (See Definition-Proposition 1 resp. Proposition 16 in \cite{EG2}) by using the homotopy property for the mixed spaces.
\end{Remark}
 
\begin{Remark}
 \label{RemAdm}
 We will frequently encounter the situation of a morphism $T \to X$ of $G$-schemes such that $T$ is open in a $G$-equivariant vector bundle over $X$. 
 We remark that, if $X$ is an admissible $G$-scheme, so is $T$. This follows since a vector bundle over a quasi-projective scheme is again quasi-projective.
\end{Remark}

\begin{Lemma}
\label{Lepull-back}
 Let $f \colon \mathscr{X} \to \mathscr{Y}$ be a flat map of quotient stacks of relative dimension $r$. Then there is a flat pull-back map
$f^* \colon A_*(\mathscr{Y}) \to A_{*+r}(\mathscr{X})$ between the Chow groups. If $\mathscr{X}$ and $\mathscr{Y}$ admit admissible 
presentations the same assertion holds for the higher Chow groups.

Furthermore, if $\mathscr{X}$ and $\mathscr{Y}$ are smooth then under the identification
$A_*(\mathscr{X})=A^*(\mathscr{X})$ the above morphism is just the natural pull-back map between the operational Chow rings. 
\end{Lemma}

\begin{proof}
 Consider presentations $\mathscr{X}=[X/G]$ and $\mathscr{Y}=[Y/H]$. By definition
$A_i(\mathscr{X})=A_{i+g}^G(X)$ with $g=\dim G$ and similar for $A_i(\mathscr{Y})$. Choose a good pair
$(V_1,U_1)$ for $G$ and a good pair $(V_2,U_2)$ for $H$. Let $l_i=\dim V_i$. As usual we will write
$X_G$ resp. $Y_H$ for the mixed space $(X \times U_1)/G$ resp. $(Y \times U_2)/H$. Consider the fibersquare
$$
\xymatrix{
Z' \ar[d] \ar[r] & Z \ar[r] \ar[d] & Y_H \ar[d] \\
X_G \ar[r] & \mathscr{X} \ar[r] & \mathscr{Y}
}
$$
Then $Z'$ is a bundle over $X_G$ resp.\ $Z$ with fiber $U_2$ resp.\ $U_1$ and $Z' \to Y_H$ is a flat map of algebraic spaces
of relative dimension $l_1+r$. Hence
$$A_{i+l_1+l_2+r}(Z')=A_{i+l_1+r}(X_G)=A_{i+r}(\mathscr{X})$$
and we define $f^*$ to be the ordinary pull-back of the flat map $Z' \to Y_H$. The exact same construction works for the higher 
equivariant Chow groups if $\mathscr{X}$ and $\mathscr{Y}$ admit admissible presentations. 

For the last part we recall that the isomorphism $A^i(\mathscr{X}) \cong A^G_{\dim X-i}(X)$ maps $c \in A^i(\mathscr{X})$ to
$c(X_G \to \mathscr{X}) \cap [X_G] \in A^G_{\dim X-i}(X)$. Thus we need to check the equality
$$ f^*(d(Y_H \to \mathscr{Y}) \cap [Y_H])=d(X_G \to \mathscr{X} \to \mathscr{Y}) \cap [X_G] $$ for $d \in A^i(\mathscr{Y})$.
This follows from the compatibility of $d$ with flat pull-backs.
\end{proof}

\subsection{Auxiliary Results}

\begin{Lemma}
\label{LeCov}
 Let $X \to Y$ be a flat morphism of schemes and $Y' \to Y$ be a finite, flat and surjective map of degree $d$.
Let $X' \to Y'$ be the base change of $X \to Y$ along $Y' \to Y$.
Assume the pull-back  $A_*(Y',m) \to A_*(X',m)$ becomes an isomorphism after inverting some integer $d'$.
Then the pull-back $A_*(Y,m) \to A_*(X,m)$ is an isomorphism after inverting $dd'$. 
\end{Lemma}

\begin{proof}
The injectivity of the pull-back $A_*(Y,m)_{dd'} \to A_*(X,m)_{dd'}$ follows from the exact diagram
$$
\xymatrix{
0 \ar[r] & A_*(Y,m)_{dd'} \ar[r] \ar[d] & A_*(Y',m)_{dd'} \ar[d]_{\cong} \\
0 \ar[r] & A_*(X,m)_{dd'} \ar[r] & A_*(X',m)_{dd'}
}
$$
and the surjectivity from the exact diagram
$$
\xymatrix{
A_*(Y',m)_{dd'} \ar[r] \ar[d]_{\cong} & A_*(Y,m)_{dd'} \ar[d] \ar[r] & 0 \\
A_*(X',m)_{dd'} \ar[r] & A_*(X,m)_{dd'} \ar[r] & 0
}
$$
where the horizonatal maps in the first diagram are induced by pull-back and in the second diagram by push-porward. The commutativity of the second diagram is 
\cite[Proposition 1.7]{Fu}.
\end{proof}

\begin{Lemma}
\label{LeFiber}
 Let $T \to X$ be a morphism of quasi-projective schemes over $k$. We assume that $X$ is equidimensional and that $T \to X$ is flat of relative dimension $a$. Let 
 $d,i \in \mathbb{Z}$ and for $x \in X$ let $h(x)$ denote the dimension of the closure of $\{x\}$ in $X$. 
 If the pull-back $A_{i-h(x)}(\Spec k(x),m)_d \to A_{i-h(x)+a}(T_x,m)_d$ is an isomorphism for every $x \in X$ and
 for any $m$, then $A_i(X,m)_d \to A_{i+a}(T,m)_d$ is an isomorphism.
\end{Lemma}

\begin{proof}
We follow Quillen's proof of the analogous result in higher K-theory (\cite[Proposition 4.1]{Qu}).
First we may assume that $X$ is irreducible for if $X=W_1 \cup \ldots \cup W_r$ is a decomposition into irreducible components we may consider the
long localization exact sequence of the pair $(W_1,X-W_1)$. By induction we are thus reduced to the irreducible case. Since the Chow groups only 
depend on the reduced structure,
we may also assume that $X$ is reduced. Let $K$ denote the function field of $X$. We have
$$A_{i-n}(\Spec K,m)= \varinjlim_U{A_i(U,m)},$$ 
$$A_{i-n+a}(T_K,m)= \varinjlim_U{A_{i+a}(T_U,m)},$$ 
where the limit goes over all non-empty open subsets of $X$ and $n$ denotes the dimension of $X$. In fact, it suffices to go over all non-empty open subsets 
with equidimensional complement, since for all non-empty 
open $U$ in $X$ there exists a non-empty open subset $U'$ contained in $U$ with equidimensional complement. We obtain a commutative diagram
$$
\xymatrix{
A_{i-n}(\Spec K,m+1) \ar[r] \ar[d] & \varinjlim_Y{A_i(Y,m)} \ar[r] \ar[d] & A_i(X,m) \ar[d] && \\
A_{i-n+a}(T_K,m+1) \ar[r] & \varinjlim_Y{A_{i+a}(T_Y,m)} \ar[r] & A_{i+a}(T,m)  &&
}
$$
$$
\xymatrix{
&&&& \ar[r] & A_{i-n}(\Spec K,m) \ar[r] \ar[d] & \varinjlim_Y{A_i(Y,m-1)} \ar[d] \\
&&&& \ar[r] & A_{i-n+a}(T_K,m) \ar[r] & \varinjlim_Y{A_{i+a}(T_Y,m-1)}
}
$$
with exact rows, where the limit goes over all proper closed equidimensional subsets of $X$.
After inverting $d$ the first and fourth vertical map become isomorphisms and we conclude by noetherian induction.
\end{proof}

\begin{Corollary}
\label{CorAffineFibers}
 Let $T \to X$ be a flat morphism of quasi-projective schemes over $k$ with fibers being affine spaces of some dimension $n$. Then the pull-back
$A_*(X,m) \to A_{*+n}(T,m)$ is an isomorphism. 
\end{Corollary}

\begin{proof}
 This is an immediate consequence of Lemma \ref{LeFiber}.
\end{proof}

\begin{Remark}
 The assertion of the above corollary in the case $m=0$ also holds without the quasi-projective assumption. One can use the same proof but using Gillet's higher Chow groups.
For his higher Chow groups a long localization exact sequence exists for arbitrary schemes. For details see Chapter 8 in \cite{Gi}.
\end{Remark}

\begin{Lemma}
 \label{LeUniSubgp}
Let $K$ be a unipotent subgroup of an algebraic group $G$ such that the quotient $G/K$ is finite of degree $d$. Then the pull-back
$A^*_G(m) \to A^*_{\{0\}}(m)$ is an isomorphism after inverting $d$.
\end{Lemma}

\begin{proof}
 Let $(V,U)$ be an admissible pair for $G$. Then $U/K \to U/G$ is a $G/K$-bundle locally trivial for the flat topology. By assumption on $G/K$ the morphism
$U/K \to U/G$ is therefore finite, flat and surjective of degree $d$. It follows that the pull-back $A_*(U/G,m) \to A_*(U/K,m) \cong A_*(U,m)$ is injective after 
inverting $d$. Also for sufficiently high degree we know that $A_*(\Spec k, m) \to A_*(U,m)$ is surjective. Since we can assume the codimension of $U^c$ in $V$
to be arbitrary high, we obtain the surjectivity of $A^*_G(m) \to A^*_{\{0\}}(m)$.
\end{proof}

\begin{Lemma}
 \label{LeGal}
Let $K/k$ be a Galois extension with Galois group $G$ and let $X$ be a scheme over $k$. Then pulling back along $X_K \to X$ induces
an isomorphism $A_*(X,m)_{\Q} \cong A_*(X_K,m)_{\Q}^G$. If $K/k$ is a finite Galois extension of degree $d$ it suffices to invert $d$.
\end{Lemma}

\begin{proof}
 We first assume that $K/k$ is finite of degree $d$. Then on the level of cycles we have an injection $z_*(X,\cdot)_d \hookrightarrow z_*(X_K,\cdot)^G_d$ since
 $X_K \to X$ is finite, flat of degree $d$. We claim that this map is also surjective. Let $W \subset X_K \times_K \Delta^r_K$ be a subvariety meeting all
faces properly. Let $S \subset G$ be the isotropy group of $W$. It suffices to see that $\sum_{g \in G/S} [gW]$ lies in $z_*(X,\cdot)_d$. For this consider
the closed subscheme $V=\cup_{g \in G/S} gW$ (equipped with the reduced structure). Then $V$ is a $G$-invariant equidimensional subscheme of $X_K \times_K \Delta^r_K$ 
that meets all faces properly. Thus it has a model $\tilde{V}$ over $k$ also meeting all faces properly. Finally all components $gW$ have the same multiplicity $1$ 
in the cycle $[V]$ and therefore $\sum_{g \in G/S} [gW]= [\tilde{V}_K]$. To complete the proof in the finite case it suffices now to note that taking $G$-invariants
is an exact functor on the category of $\Z[\frac{1}{d}]$-modules with $G$-action, hence $H_i(z_*(X_K,\cdot)_d^G)=H_i(z_*(X_K,\cdot))_d^G$. 
The general case follows from the finite case and the fact that $A_*(X_K,m)^G=\varinjlim_{L/k} A_*(X_L,m)^{G(L/k)}$, where the limit goes over all 
finite Galois subextensions $L/k$ of $K$.
\end{proof}

\subsection{A Pull-Back Lemma}
Throughout we consider the situation of an exact sequence 
$$
\xymatrix{
0 \ar[r] & A \ar[r] & G \ar[r] & H \ar[r] & 0
}
$$
of algebraic groups and an admissible $H$-scheme $X$ such that the induced $G$-action on $X$ makes $X$ also into an admissible $G$-scheme.
These conditions are always satisfied if $X$ is quasi-projective and normal by Lemma \ref{LeAdm}.
We are then interested in properties of the pull-back homomorphism (Lemma \ref{Lepull-back})
$$
A_*([X/H],m) \to A_*([X/G],m).
$$

\begin{Proposition}
\label{PropExt1}
 Let 
$$
\xymatrix{
0 \ar[r] & A \ar[r] & G \ar[r] & H \ar[r] & 0
}
$$
be an exact sequence of algebraic groups and $X$ an admissible $H$-scheme such that the induced $G$-action makes $X$ also
into an admissible $G$-scheme. We also assume $H$ to be special. 

Let $d \in \mathbb{Z}$ such that $A^*_{A_L}(m) \to A^*_{\{0\}}(m)$ becomes an isomorphism after inverting $d$ for every field extension $L$ of $k$ and every $m$. Then
the pull-back $A_*([X/H],m) \to A_*([X/G],m)$ becomes an isomorphism after inverting $d$.
\end{Proposition}

\begin{proof}
First note that the natural map $[X/G] \to [X/H]$ is flat of relative dimension $-a$ with $a=\dim A$. 
We can choose for any $i \in \mathbb{Z}$ an admissible pair $(V,U)$ for the $H$-action such that 
$A_{j+l}([(X \times U)/G],m)=A_j([X/G],m)$ and $A_{j+l}((X \times U)/H,m)=A_j([X/H],m)$ for all $j>i$. 
Here $l$ denotes the dimension of $V$. Note that $X \times U$ is again an admissible $G$-scheme (cf. Remark \ref{RemAdm}). 
Replacing $X$ by $X \times U$ we may thus assume that $[X/H]$ 
is a quasi-projective scheme. 

Let now $(X \times U)/G$ be a quasi-projective mixed space for $G$.
Let $\bar{U}$ be the quotient $U/A$. Then we can identify $(X \times U)/G$ with the quotient $(X \times \bar{U})/H$ and under this 
identification the map $(X \times U)/G \to X/H$ corresponds to the 
$\bar{U}$-bundle $(X \times \bar{U})/H \to X/H$. It is Zariksi locally trivial since $H$ is special. 
We are left to show that the pull-back of this map is an isomorphism after inverting $d$. 
This will follow from Lemma \ref{LeFiber} once we have seen that the pull-back $A_{j-h(x)}(\Spec k(x),m)_d \to A_{j-h(x)+l-a}(\bar{U}_{k(x)},m)_d$ is an isomorphism
for every $x \in X/H$. Here $h(x)$ is the dimension of the closure of $\{x\}$ in $X/H$. Let us write $L=k(x)$. 
Assuming the codimension of $U^c$ in $V$ to be sufficiently large we obtain by assumption
$$
A_{j-h(x)}(\Spec L,m)_d=A_{j-h(x)+l}(U_L,m)_d=A_{j-h(x)+l-a}(\bar{U}_L,m)_d.
$$
For this recall $A_{j+l-a}(\bar{U}_L,m)=A^{A_L}_j(m)$ and $A_{j+l}(U_L,m)=A^{\{0\}}_j(m)$. This proves the claim.
\end{proof}

The above proposition applies to the following cases.
\begin{Corollary}
\label{CorExt1}
 In the situation of Proposition \ref{PropExt1} the following assertions hold.
\begin{equivlist}
 \item If $A$ is unipotent then $A_*([X/H],m) \to A_*([X/G],m)$ is an isomorphism.
 \item If $A$ is finite of degree $d$ then $A_*([X/H],m) \to A_*([X/G],m)$ becomes an isomorphism after inverting $d$.
\end{equivlist}
\end{Corollary}

\begin{proof}
 The first part follows from Corollary \ref{CorAffineFibers} and the second part follows from Lemma \ref{LeUniSubgp} applied to the case $K=\{0\}$.
\end{proof}

The assumption on $H$ to be special is crucial for the proof of Proposition \ref{PropExt1}, since we need to know that the fibers of the $\bar{U}$-bundle 
$(X \times \bar{U})/H \to X/H$ appearing in the proof are given by $\bar{U}$ in order to apply Lemma \ref{LeFiber}. However, we have the following version when $H$ is finite.

\begin{Proposition}
\label{PropExt2}
 Let 
$$
\xymatrix{
0 \ar[r] & A \ar[r] & G \ar[r] & H \ar[r] & 0
}
$$
be an exact sequence of algebraic groups and $X$ an admissible $H$-scheme such that the induced $G$-action makes $X$ also
into an admissible $G$-scheme. We assume that $H$ is finite of degree $d$. 

Let $d' \in \mathbb{Z}$ such that $A^*_{A_L}(m) \to A^*_{\{0\}}(m)$ becomes an isomorphism after inverting $d'$ for every field extension $L$ of $k$ and any $m$. Then
the pull-back $A_*([X/H],m) \to A_*([X/G],m)$ becomes an isomorphism after inverting $dd'$.
\end{Proposition}

\begin{proof}
We argue the same way as in Proposition \ref{PropExt1} and then have to see that the pull-back of 
$(X \times \bar{U})/H \to X/H$ becomes an isomorphism after inverting $dd'$. As mentioned earlier we cannot apply Lemma \ref{LeFiber}
since the above $\bar{U}$-bundle is not locally trivial for the Zariski topology. Instead it becomes trivial after the finite, flat and surjective base change
$X \to X/H$ of degree $d$, i.e. there is a cartesian diagram
$$
\xymatrix{
X \times \bar{U} \ar[r] \ar[d] & X \ar[d] \\
(X \times \bar{U})/H \ar[r] & X/H.
}
$$
The claim thus follows from Lemma \ref{LeCov}.
\end{proof}

\begin{Corollary}
\label{CorExt2}
 In the situation of Proposition \ref{PropExt2} the following assertions hold.
\begin{equivlist}
 \item If $A$ is unipotent then $A_*([X/H],m)_d \to A_*([X/G],m)_d$ is an isomorphism.
 \item If $A$ is finite of degree $d'$ then $A_*([X/H],m)_{dd'} \to A_*([X/G],m)_{dd'}$ is an isomorphism.
\end{equivlist}
\end{Corollary}

In the next proposition we show that the assertion of Proposition \ref{PropExt1} is valid over $\Q$ for arbitrary $H$.

\begin{Proposition}
 \label{PropExtrat}
Let 
$$
\xymatrix{
0 \ar[r] & A \ar[r] & G \ar[r] & H \ar[r] & 0
}
$$
be an exact sequence of algebraic groups and $X$ an admissible $H$-scheme such that the induced $G$-action makes $X$ also
into an admissible $G$-scheme.  

Assume $A^*_{A_L}(m)_{\Q} \to A^*_{\{0\}}(m)_{\Q}$ is an isomorphism for every field extension $L$ of $k$ and any $m$. Then
the pull-back $A_*([X/H],m)_{\Q} \to A_*([X/G],m)_{\Q}$ is an isomorphism.
\end{Proposition}

\begin{proof}
 Using the notation of the proof of Proposition \ref{PropExt1} we need to see that the pull-back of the $\bar{U}$-bundle $T:=(X \times \bar{U})/H \to X/H$ is an 
isomorphism over $\Q$. It suffices to see that $A_*(\Spec k(x),m)_{\Q} \to A_*(T_x,m)_{\Q}$ is an isomorphism for $x \in X/H$.
The above $\bar{U}$-bundle may not be trivial for the Zariski topology, but we still have $T_{\bar{x}}=\bar{U}_{\bar{x}}$ and thus
$A_*(\Spec k(x)^{sep},m)_{\Q} \to A_*(T_{\bar{x}},m)_{\Q}$ is an isomorphism by assumption. The claim then follows from Lemma \ref{LeGal} and the fact that the Galois action 
is compatible with pull-back.
\end{proof}

\begin{Corollary}
\label{CorExtrat}
 In the situation of Proposition \ref{PropExtrat} the following assertions hold.
\begin{equivlist}
 \item If $A$ is unipotent then $A_*([X/H],m)_{\Q} \to A_*([X/G],m)_{\Q}$ is an isomorphism.
 \item If $A$ is finite then $A_*([X/H],m)_{\Q} \to A_*([X/G],m)_{\Q}$ is an isomorphism.
\end{equivlist}
\end{Corollary}

\begin{Lemma}
\label{LeUni2}
Let $G$ be a split extension 
 $$
 \xymatrix{
 0 \ar[r] & K \ar[r] & G \ar[r] & H \ar[r] & 0
 }
 $$
of an algebraic group $H$ by a unipotent group $K$. Choose a splitting $H \hookrightarrow G$ and let $X$ be a normal, quasi-projective $G$-scheme. 
Then the pull-back map
$$
A_*^G(X,m)_{\Q} \to A_*^H(X,m)_{\Q}
$$ 
is an isomorphism. If $G$ is special, this above map is an isomorphism over $\Z$.
\end{Lemma}

\begin{proof}
Let $(V,U)$ be an admissible pair for the $G$-action on $X$. It follows from the proof of Lemma \ref{LeAdm} that $(V,U)$ is then also admissible for the induced $H$-action.
The morphism $(X \times U)/H \to (X \times U)/G$ is a $G/H$-bundle. If $G$ is special this bundle is locally trivial for the Zariski topology.
Hence the lemma follows from Corollary \ref{CorAffineFibers} in the special case and Lemma \ref{LeGal} and \ref{LeFiber} in the general case.
\end{proof}

\subsection{The Restriction Map} We want to describe  properties of the restriction map $res^G_T \colon A^G_*(X) \to A^T_*(X)$, where
$T$ is a split torus in $G$. This map is defined via flat pull-back of the natural map $X_T \to X_G$  between the mixed spaces.  
Note that more generally one has a restriction map $res^G_H \colon A^G_*(X) \to A^H_*(X)$ for every subgroup $H$ of $G$. 
We will need the following result.

\begin{Theorem}
\label{ThInv}
Let $G$ be a connected reductive group with split maximal torus $T$ and Weyl group $W=W(G,T)$. Let $X$ be a $G$-scheme.
\begin{equivlist}
\item $W$ acts on $A^T_*(X)$. Furthermore, the restriction morphism
$A^G_*(X) \to A^T_*(X)$ induces a map $r \colon A^G_*(X) \to A^T_*(X)^W$. 
\item Assume $X$ is smooth. Then $r$ is an isomorphism after tensoring with $\mathbb{Q}$.
\end{equivlist}
\end{Theorem}

Part (iii) is basically proved in \cite{EG}, where Edidin and Graham consider the case $X=\Spec k$.
However, there seems to be no complete proof of part (ii) in the literature. We therefore give a proof. 

In the following $A^*(X;\Q)$ will denote the operational Chow ring of $X$ consisting
of characteristic classes with values in rational Chow groups, i.e. an element $c \in A^*(X;\Q)$ assigns
to each $T \to X$ a morphism 
$$c(T \to X)\colon A_*(T)_{\Q} \to A_*(T)_{\Q}$$
satisfying the usual compatibility conditions (\cite[Section 17.1]{Fu}). 
A proper map $\pi \colon \tilde{X} \to X$ is called an envelope if for each irreducible subspace $V \subset X$ there
exists an irreducible subspace $\tilde{V} \subset \tilde{X}$ sucht $\pi$ maps $\tilde{V}$ birationally onto $V$.

\begin{Remark}
 There is a natural map $A^*(X)_{\Q} \to A^*(X;\Q)$ and this map is an isomorphism if $X$ is smooth. This follows from
 $$
 \xymatrix{
 A^*(X)_{\Q} \ar[r]_{\cong}^{\cap[X]} \ar[d] & A_*(X)_{\Q} \ar@{=}[d] \\
 A^*(X;\Q) \ar[r]_{\cong}^{\cap [X]} & A_*(X)_{\Q}.
 }
 $$
\end{Remark}

We recall the following easy lemma.

\begin{Lemma}
\label{Lepullbackenvelope}
 \begin{equivlist}
  \item Let $\pi \colon \tilde{X} \to X$ be a proper surjective map. Then $\pi_* \colon A_*(\tilde{X})_{\Q} \to A_*(X)_{\Q}$ is surjective
  and $\pi^* \colon A^*(X;\Q) \to A^*(\tilde{X};\Q)$ is injective.
  \item Let $\pi \colon \tilde{X} \to X$ be a birational envelope. Then $\pi_* \colon A_*(\tilde{X}) \to A_*(X)$ is surjective
  and $\pi^* \colon A^*(X) \to A^*(\tilde{X})$ is injective.
 \end{equivlist}
\end{Lemma}

\begin{proof}
 The first part of (i) is \cite[Proposition 1.3]{Ki}. The first part of (ii) follows immediately from the definition of an envelope.
 The second part of (i) and (ii) are formal consequences of their first parts. 
\end{proof}

\begin{Lemma}
 \label{LeG/B-Bundle}
 Let $G$ be a connected reductive group with split maximal torus $T$ and Weyl group $W=W(G,T)$. 
 Let $M$ be smooth and $E \to M$ be a principal $G$-bundle. Consider a Borel subgroup $B \supset T$.
 Then $W$ acts on $A^*(E/B)$ and pull-back induces an isomorphism $A^*(M)_{\Q} \cong A^*(E/B)^W_{\Q}$.
\end{Lemma}

\begin{Remark}
 This lemma is also mentioned (without proof) in \cite[Section 2.5]{Vi}.
\end{Remark}

\begin{proof}
  We identify $W=N_G(T)/T$ and choose $w \in N_G(T)$. Then $w$ induces an automorphism $w \colon E/T \to E/T$. 
This defines an action of $W$ on $A^*(E/T)=A^*(E/B)$. Since $w$ lies in $G$
the diagram 
$$
\xymatrix{
E/T \ar[r] \ar[d]_w & E/G=M  \\
E/T \ar[ur]
}
$$
commutes and this implies that the image of the pull-back $A^*(M) \to A^*(E/B)$ lies in $A^*(E/B)^W$. We are left to show that
$$A^*(M)_{\Q} \to A^*(E/B)^W_{\Q}$$
is an isomorphism. 
Let us first show that $A_*(M)_{\Q} \to A_*(E/B)^W_{\Q}$ is surjective. For this the smoothness assumption on $M$ is not needed.
We recall that every $G$-torsor is locally isotrivial by \cite[XIV Lemma 1.4]{Ra}. This means that there exists a covering of $M$ by open subsets $U$
with the property that for each $U$ there is a finite, etale and surjective map $U' \to U$ such that $E_{U'}=E \times_M U' \to U'$ becomes a trivial $G$-torsor.
Let $V$ denote the complement of such an $U$ in $M$ and consider the commutative diagram
$$
\xymatrix{
A_*(V)_{Q} \ar[r] \ar[d] & A_*(M)_{\Q} \ar[r] \ar[d] & A_*(U)_{\Q} \ar[r] \ar[d] & 0 \\
A_*(E_V/B)^W_{\Q} \ar[r] & A_*(E/B)^W_{\Q} \ar[r] & A_*(E_U/B)^W_{\Q} \ar[r]  & 0 
}
$$
with exact rows. An easy diagram chase shows that if the first and last vertical map are surjective so is $A^*(M)_{\Q} \to A^*(E/B)^W_{\Q}$.
Using noetherian induction we are thus reduced to the case that there exists a proper surjective map $M' \to M$ such that
$E_{M'} \to M'$ is trivial. Since the diagramm
$$
\xymatrix{
A_*(M')_{\Q} \ar[r] \ar[d] & A_*(E_{M'}/B)^W_{\Q} \ar[d] \\
A_*(M)_{\Q} \ar[r] & A_*(E/B)^W_{\Q}
}
$$
commutes (\cite[Proposition 1.7]{Fu}) and since $A_*(E_{M'}/B)^W_{\Q} \to A_*(E/B)^W_{\Q}$ is surjective by part (i) of the previous lemma we are further reduced to the case
of a trivial $G$-torsor $E=G \times M \to M$.
Since $G/B$ has a decomposition into affine cells we obtain in this case $A_*(E/B)_{\Q}=A_*(G/B)_{\Q} \otimes A_*(M)_{\Q}$.
From \cite[Section 8]{De} we get $A_*(G/B)_{\Q}=S_{\Q}/(S_+^W)$, where $S=\Sym(\hat{T})$ and $S_+^W$ denotes the submodule generated by homogeneous $W$-invariant
elements of positive degree. Since $(S_{\Q}/(S^W_+))^W = \Q$ we obtain $A_*(E/B)^W_{\Q}=A_*(M)$ as wanted. 

By the previous lemma we know that $A^*(M;\Q) \to A^*(E/B;\Q)$ is injective but since $M$ (and therefore $E$) is smooth we obtain the injectivity of
$A^*(M)_{\Q} \to A^*(E/B)_{\Q}$.
\end{proof}

\begin{proof}(of Theorem \ref{ThInv})
The assertion (i) and (ii) are immediate consequences of Lemma \ref{LeG/B-Bundle}. 
Under the assumption that $A^*_T(X)$ is $\Z$-torsion free the surjectivity of $r$ follows from part (ii) 
by using the argumentation of the proof of Lemma 5 in \cite{EG}. 
\end{proof}

\section{The Chow Ring of the Stack of level-$n$ Barsotti-Tate Groups}
\subsection{The Stack of truncated Displays}
Let $R$ be an $\mathbb{F}_p$-algebra. We denote by $W_n(R)$  the ring of truncated Witt vectors of length $n$. Let $I_{n,R} \subset W_n(R)$
be the image of the Verschiebung $W_{n-1}(R) \to W_n(R)$ and $J_{n,R} \subset W_n(R)$ be the kernel of the projection
$W_n(R) \to W_{n-1}(R)$. The Frobenius on $R$ induces a ring homomorphism $\sigma \colon W_{n}(R) \to W_n(R)$ and
the inverse of the Verschiebung induces a bijective $\sigma$-linear map $\sigma_1 \colon I_{n+1,R} \to W_n(R)$.
Note that $pR=0$ implies $I_{n,R}J_{n,R}=0$, hence we may view $I_{n+1,R}$ as a $W_n(R)$-module. 

Truncated displays were introduced in \cite{La}. Let us recall the necessary notations.
We are only going to need the following description of truncated displays.
\begin{Definition}
 A truncated display of level $n$ over an $\mathbb{F}_p$-algebra $R$ is a triple $(L,T,\Psi)$ consisting
of projective $W_n(R)$-modules $L$ and $T$ of finite rank and a $\sigma$-linear automorphism $\Psi  \colon L \oplus T \to L \oplus T$.
\end{Definition}
A morphism between truncated displays is defined as follows. First we can use $\Psi$ to define $\sigma$-linear maps
$$F \colon L \oplus T \to L \oplus T, \quad l+t \mapsto p\Psi(l)+\Psi(t),$$
$$F_1 \colon L \oplus (T \otimes_{W_n(R)} I_{n+1,R}) \to L \oplus T, \quad l+(t \otimes \omega) \mapsto \Psi(l)+\sigma_1(\omega)\Psi(t).$$
Then a morphism between two truncated displays $(L,T,\Psi)$ and $(L',T',\Psi')$ of level $n$ is given by a matrix 
$
\left(
  \begin{array}{ c c }
     A & B \\
     C & D
  \end{array} \right),
$
where $A \in \Hom(L,L')$, $B \in \Hom(T,L')$, $C \in  \Hom(L,T' \otimes_{W_n(R)} I_{n+1,R})$ and $D \in \Hom(T,T')$
such that 
$$
\xymatrix{ L \oplus T \ar[r]^{F} \ar[d] & L \oplus T \ar[d] \\ L' \oplus T' \ar[r]^{F'} & L' \oplus T'}
\hspace{0.3in}
\xymatrix{ L \oplus (T \otimes_{W_n(R)} I_{n+1,R}) \ar[r]^-{F_1} \ar[d] & L \oplus T \ar[d] \\ 
L' \oplus (T' \otimes_{W_n(R)} I_{n+1,R}) \ar[r]^-{F_1'} & L' \oplus T'}
$$
commute. 

The height of a truncated display is defined as the rank of $L \oplus T$ and the dimension as the rank of $T$. 
Both are locally constant functions on $\Spec R$.
Let $\disp_n \to \Spec \mathbb{F}_p$ denote the stack of truncated displays of level $n$. That is for 
$R$ an $\mathbb{F}_p$-algebra $\disp_n(\Spec R)$ is the groupoid of truncated displays of level $n$. 
It is proved in \cite[Proposition 3.15]{La} that $\disp_n$ is a smooth Artin algebraic stack of dimension 
zero over $\mathbb{F}_p$ with affine diagonal. 

For $h \in \mathbb{N}$ and $0 \leq d \leq h$ we denote by $\disp^{h,d}_n$ the open and closed substack
of truncated displays of level $n$ with constant height $h$ and constant dimension $d$. Then 
$$ \disp_n=\coprod_{h,d} \disp^{h,d}_n. $$

\textit{A Presentation of $\disp_n^{h,d}$.} We will adopt the notation
of the proof of Proposition 3.15 in \cite{La}. Let $X_n^{h,d}$ be the functor on affine $\mathbb{F}_p$-schemes
with $X_n^{h,d}(R)=\GL_h(W_n(R))$. This is an affine open subscheme of $\mathbb{A}^{nh^2}$. Furthermore, let 
$G_n^{h,d}$ be the functor such that $G_n^{h,d}(R)$ is the group of invertible matrices 
$\left(
  \begin{array}{ c c }
     A & B \\
     C & D
  \end{array} \right)
$
with $A \in \GL_{h-d}(W_n(R))$, $B \in \Hom(W_n(R)^d,W_n(R)^{h-d})$, $C \in \Hom(W_n(R)^{h-d},I_{n+1,R}^d)$ and $T \in \GL_d(W_n(R))$.
Then $G_n^{h,d}$ is a connected algebraic group of dimension $nh^2$. 

\begin{Remark}
 \label{RemG_1^{h,d}}
Since $I_{2,R}$ is in bijection to $R$ via $\sigma_1$ we may view $G_1^{h,d}(R)$ as the group of invertible matrices with entries in $R$ with respect to the
multiplication given by 
$$
\left(
  \begin{array}{ c c }
     A & B \\
     C & D
  \end{array} \right)
\left(
  \begin{array}{ c c }
     A' & B' \\
     C' & D'
  \end{array} \right)
=
\left(
  \begin{array}{ c c }
     AA' & AB'+BD' \\
     C\sigma(A')+\sigma(D)C' & DD'
  \end{array} \right),
$$
where in the four blocks we have the usual matrix multiplication.  
\end{Remark}

Let  $\pi_n^{h,d} \colon X_n^{h,d} \to  \mathcal{D}$isp$_{n,d}$ be the functor that assigns to an invertible matrix 
$\Psi \in \GL_h(W_n(R))$ the truncated display $(W_n(R)^{h-d},W_n(R)^d,\Psi)$, where we view $\Psi$ as a $\sigma$-linear map $W_n(R)^h \to W_n(R)^h$ 
via $x \mapsto \Psi \cdot \sigma x$.
Now if we let $G_n^{h,d}$ act on $X_n^{h,d}$
via $$G \cdot \Psi=G \Psi \sigma_1(G)^{-1}$$ where $\sigma_1(G)=\left(
  \begin{array}{ c c }
     \sigma(A) & p\sigma(B) \\
     \sigma_1(C) & \sigma(D)
  \end{array} \right),$
then every $G \in G_n^{h,d}$ defines an isomorphism $\pi_n^{h,d}(\Psi) \to \pi_n^{h,d}(G \cdot \Psi)$ of truncated displays.
On the contrary if $G$ defines an isomorphism $\pi_n^{h,d}(\Psi) \to \pi_n^{h,d}(\Psi')$ then necessarily $\Psi'=G \Psi \sigma_1(G)^{-1}$. We thus obtain

\begin{Theorem}
\label{ThRepDisp}
The functor $\pi_n^{h,d}$ induces an isomorphism of stacks $$ [X_n^{h,d}/G_n^{h,d}] \cong \disp^{h,d}_n.$$ 
\end{Theorem}

There are the following two obvious vector bundles on $\disp_n^{h,d}$.
\begin{Definition}
\label{DefVB}
Let $\Spec R \to \disp_n^{h,d}$ be a map corresponding to a 
truncated display $\mathcal{P}=(L,T,\Psi)$. 
\begin{equivlist}
 \item We denote by $\mathcal{L}ie$ the vector bundle of rank $d$ over $\disp_n^{h,d}$ that assigns to $\Spec R \to \disp_n^{h,d}$ the vector bundle 
$\Lie(\mathcal{P})=T/I_{n,R}T$ of rank $d$ over $R$. 
 \item By $^t \mathcal{L}ie^{\vee}$ we denote the vector bundle of rank $h-d$ that assigns to $\Spec R \to \disp_n^{h,d}$ the vector bundle $L/I_{n,R}L$
of rank $h-d$ over $R$.
 \end{equivlist}
\end{Definition}

\begin{Remark}
\label{RemVB}
The notation $^t \mathcal{L}ie^{\vee}$ in the above definition stems from the fact that the dual of $L/I_{n,R}L$ gives the Lie algebra of 
the dual display $\mathcal{P}^t$. For the definition of the dual display see \cite[Definition 19]{Zi}.
\end{Remark}

\textit{The Truncated Display Functor.} As already mentioned in the introduction the strategy for computing the Chow ring of the stack of truncated Barsotti-Tate groups 
is to relate it to the stack of truncated displays. This happens via the truncated display functor 
$$\phi_n \colon BT_n \to \disp_n$$
constructed in \cite{La}. Let us briefly sketch the construction. 

Let $G$ be a $p$-divisible group over an $\mathbb{F}_p$-algebra $R$. The ring of Witt vectors $W(R)$ is $p$-adically complete and the ideal $I_R$ in $W(R)$ carries natural divided powers compatible with the canonical divided powers of $p$. Let $\mathbb{D}(G)$ denote the covariant Dieudonne crystal of $G$.  We can
evaluate $\mathbb{D}(G)$ at $W(R) \to R$ and set $P=\mathbb{D}(G)_{W(R) \to R}$ and $Q=\Ker(P \to \Lie(G))$. Furthermore,
let $F^{\sharp} \colon P^{\sigma} \to P$ and $V^{\sharp} \colon P \to P^{\sigma}$ be the maps induced by Frobenius
and Verschiebung of $G$. One can show that there are $\sigma$-linear maps $F \colon P \to P$ resp.\ $\dot{F} \colon Q \to P$ compatible with base change in $R$
such that $(P,Q,F,\dot{F})$ is a display which induces the maps $F^{\sharp}$ and $V^{\sharp}$. See \cite[Proposition 2.4]{La} for the precise statement.
This construction yields a 1-morphism
$$
\phi \colon BT \to \disp
$$
from the stack of Barsotti-Tate groups to the stack of displays.
It is clear from the construction that the Lie algebra of $G$ is equal to the Lie algebra of $\phi(G)$ defined by $P/Q$.  

Moreover, one can prove that for all $n$ there are maps $\phi_n \colon BT_n \to \disp_n$ compatible with the truncation
maps on both sides such that $\phi$ is the projective limit of the system $(\phi_n)_{n \geq 1}$. The central result in \cite{La} is that
$\phi_n$ is a smooth morphisms of smooth algebraic stacks over $\mathbb{F}_p$ which is an equivalence on
geometric points. 

\subsection{Group Theoretic Properties of $G_n^{h,d}$} 
We denote by $K_{(n,m)}^{h,d}$ the kernel of the projection $G_n^{h,d} \to G_m^{h,d}$ for $m<n$ and by 
$\tilde{K}^{h,d}_n$ the kernel of the projection $G_n^{h,d} \to \GL_{h-d} \times \GL_d$. Note that $G_n^{h,0}=\GL_{h}(W_n(\cdot))$.
We recall the following well known facts about the Witt ring.
For an $\F_p$-algebra $R$ we denote by $[\cdot] \colon R \to W_n(R)$ the map $r \mapsto (r,0,\ldots,0)$ and ${}^V(\cdot) \colon W(R) \to W(R)$ is the Verschiebung.
\begin{Lemma}
 Let $R$ be an $\mathbb{F}_p$-algebra and $x,y \in R$. Then $[x+y]-[x]-[y]$ lies in $^VW(R)$. 
Furthermore, ${}^{V^r}W(R) \cdot {}^{V^s}W(R) \subset {}^{V^{r+s}}W(R)$.
\end{Lemma}

\begin{proof}
The first part follows immediately from the fact that ${}^VW(R)$ is the kernel of the ring homomorphism 
$\mathbb{W}_0 \colon W(R) \to R$ and the fact $\mathbb{W}_0([x])=x$ for all $x \in R$. 

For the second part we may assume $r \geq s$. We then write ${}^{V^r}x {}^{V^s}y={}^{V^r}(x{}^{F^rV^s}y)=p^s \cdot {}^{V^r}(x{}^{F^{r-s}}y)$. Since $pR=0$ we have
$p(x_0,x_1,\ldots)=(0,x_0^p,x_1^p,\ldots)$ in $W(R)$ and the lemma follows.
\end{proof}

\begin{Lemma}
\label{LeG_n^{h,d}}
\begin{equivlist}
 \item $K_{(n,m)}^{h,d}$ is unipotent. 
 \item $\tilde{K}^{h,d}_n$ is unipotent.
\end{equivlist}
\end{Lemma}

\begin{proof}
(i) First note that $K_{(n,n-1)}^{h,0}=ker(\GL_h(W_n(\cdot)) \to \GL_h(W_{n-1}(\cdot)))$ is unipotent. To see this we consider the Verschiebung
${}^V(\cdot)$ as a map $W_n(R) \to W_n(R)$. Then by the above lemma the map
$$
\mathbb{G}_a^{h^2} \to K_{(n,n-1)}^{h,0}, \quad A \mapsto I_h + {}^{V^{n-1}}[A]
$$
is an isomorphism of algebraic groups. 

Next we show that $K_{h,d}^{(n,n-1)}$ is unipotent. This is the group of matrices 
$
\left(
  \begin{array}{ c c }
     A & B \\
     C & D
  \end{array} \right)
$
with $A \in K^{h-d,0}_{(n,n-1)}$, $B \in J_n^{(h-d) \times d}$, $C \in J_{n+1}^{d \times (h-d)}$ and $D \in K^{d,0}_{(n,n-1)}$.
The multiplication in this group is given by
$$
\left(
  \begin{array}{ c c }
     A & B \\
     C & D
  \end{array} \right)
\left(
  \begin{array}{ c c }
     A' & B' \\
     C' & D'
  \end{array} \right)
=
\left(
  \begin{array}{ c c }
     AA' & AB'+BD' \\
     CA'+DC' & DD'
  \end{array} \right)
$$
Starting with the normal subgroup 
$
\left(
  \begin{array}{ c c }
     I_{h-d} & J_n^{(h-d) \times d} \\
     J_{n+1}^{d \times (h-d)} & I_d
  \end{array} \right),
$
which is isomorphic to $\mathbb{G}_a^{2d(h-d)}$, and then using the fact that $K_{(n,n-1)}^{h-d,0}$ resp.\ $K_{(n,n-1)}^{d,0}$ are isomorphic to $\mathbb{G}_a^{(h-d)^2}$ resp.
$\mathbb{G}_a^{d^2}$ one obtains a filtration of $K^{h,d}_{(n,n-1)}$ by normal subgroups, whose successive quotients are isomorphic to a product of copies of $\mathbb{G}_a$.
Now we have an exact sequence 
$$
\xymatrix{
0 \ar[r] & K^{h,d}_{(n,n-1)} \ar[r] & K_{(n,m)}^{h,d} \ar[r] & K^{h,d}_{(n-1,m)} \ar[r] & 0
}
$$
and by induction we may assume that $K^{h,d}_{(n-1,m)}$ is unipotent. It follows that $K_{(n,m)}^{h,d}$ is unipotent. \\
(ii) For $n=1$ the assertion is obvious in view of Remark \ref{RemG_1^{h,d}}. For $n>1$ we use the exact sequence 
$$
\xymatrix{
0 \ar[r] & K^{h,d}_{(n,n-1)} \ar[r] & \tilde{K}^{h,d}_n \ar[r] & \tilde{K}^{h,d}_{n-1} \ar[r] & 0.
}
$$
By induction and part (i) it follows that $\tilde{K}^{h,d}_n$ is unipotent. 
\end{proof}

\begin{Corollary}
\label{CorG_n^{h,d}}
 \begin{equivlist}
  \item $G_n^{h,d}$ is special.
  \item $\tilde{K}^{h,d}_n$ is the unipotent radical of $G_n^{h,d}$.
  \item The projection $X_n^{h,d} \to X_1^{h,d}$ is a trivial $K^{h,0}_{(n,1)}$-torsor.
 \end{equivlist}
\end{Corollary}

\begin{proof}
We have the exact sequence
$$
\xymatrix{
0 \ar[r] & \tilde{K}^{h,d}_n \ar[r] & G_n^{h,d} \ar[r] & \GL_{h-d} \times \GL_d \ar[r] & 0.
}
$$
Now $\tilde{K}^{h,d}_n$ is unipotent, thus special. Since $\GL_{h-d} \times \GL_{d}$ is also special part (i) follows. 

Clearly the projection $X_n^{h,d} \to X_1^{h,d}$ is a $K^{h,0}_{(n,1)}$-torsor by definition of $K^{h,0}_{(n,1)}$.
It is trivial since $K^{h,0}_{(n,1)}$ is unipotent and $X_1^{h,d}$ is affine.
\end{proof}

\subsection{The Chow Ring of $\disp_n$}
We start with the following result which reduces the calculation of $A^*(\disp_n)$ to the case $n=1$.
\begin{Theorem}
\label{Thhigherdisp}
 The pull-back 
$$
\tau_n^* \colon A^*(\disp^{h,d}_1) \to A^*(\disp_n^{h,d})
$$
of the truncation $\tau_n \colon \disp_n^{h,d} \to \disp^{h,d}_1$ is an isomorphism.
\end{Theorem}

\begin{proof}
 Under the presentation $\disp_n^{h,d}=[X_n^{h,d}/G_n^{h,d}]$ the truncation $\tau_n$ is induced by the natural projections
$X_n^{h,d} \to X_1^{h,d}$ and $G_n^{h,d} \to G_1^{h,d}$. Thus $\tau_n$ factors as
$$[X_n^{h,d}/G_n^{h,d}] \to [X_1^{h,d}/G_n^{h,d}] \to [X_1^{h,d}/G_1^{h,d}].$$
The pull-back of the second map is an isomorphism by Lemma \ref{LeG_n^{h,d}} and
Corollary \ref{CorExt1}. 
To show that the pull-back of the first map is also an isomorphism let us abbreviate $X=X_1^{h,d}$ and $G=G_n^{h,d}$. By part (iii) of Corollary \ref{CorG_n^{h,d}} we know that $X_n^{h,d}=X \times K$ with $K=K^{h,0}_{(n,1)}$, and the projection $X \times K \to X$ is $G$-equivariant.
Moreover, $K$ is an affine space by Lemma \ref{LeG_n^{h,d}}. After replacing $[X/G]$ by an appropiate mixed space (cf. Convention \ref{ConMixedSpace}), i.e.
replacing $X$ by $X \times U$ where $(V,U)$ is an admissible pair with high codimension,
we may assume that $[X/G]$ is a quasi-projective scheme.
We claim that $(X \times K)/G \to X/G$ is a Zariksi locally-trivial affine bundle. Since $G$ is special by part (i) of Corollary \ref{CorG_n^{h,d}} the 
principal $G$-bundle $X \to X/G$ is locally trivial for the Zariski topology and after replacing $X/G$ by an appropiate open subset we may assume
$X=G \times (X/G)$. We then have an isomorphism $(G \times (X/G) \times K)/G \cong (X/G) \times K$ given by the assignment $(g,x,k) \mapsto (x,k')$,
where $k'$ is defined by $g^{-1}(g,x,k)=(1,x,k')$. This proves the claim and hence the pull-back of the first map is also an 
isomorphism by Corollary \ref{CorAffineFibers}.
\end{proof}

The main ingredient of the computation of $A^*\disp_1^{h,d}$ is the following Proposition

\begin{Proposition}
\label{PropChowConj}
Let $G$ be a connected split reductive group over a field $k$ with split maximal torus $T$. 
Consider an isogeny $\varphi \colon L \to M$, where $L$ and $M$ are Levi components of parabolic subgroups $P$ and $Q$ of $G$.
Assume $T \subset L$ and let $g_0 \in G(k)$ such that $\varphi(T)={}^{g_0}T$.
Let $\tilde{\varphi} \colon T \to T$ denote the isogeny $\varphi$ followd by conjugation with $g_0^{-1}$.
We write $S=\Sym(\hat{T})=A^*_T$ and $S_+=A^{\geq 1}_T$. We have a natural action of $\tilde{\varphi}$ on $S$, 
that we will also denote by $\tilde{\varphi}$. 

Consider the action of $L$ on $G$ by $\varphi$-conjugation. 
If $W_G=W(G,T)$ and $W_L=W(L,T)$ denote the respective Weyl groups we have
$$
A^*_L(G)_{\Q}=S^{W_L}_{\Q}/(f-\tilde{\varphi} f \mid f \in S_+^{W_G})_{\Q}.
$$
If $G$ is special we have
$$
 A^*_L(G)=S^{W_L}/(f-\tilde{\varphi} f \mid f \in S_+^{W_G}).
$$
(Note that the action of $\tilde{\varphi}$ on $S^{W_G}$ is independent of the choice of $g_0$ since two choices differ by an element of $N_G(T)$.)
\end{Proposition}

\begin{proof}
The case of special $G$ is proven in \cite[Proposition 1.1]{Br}.
It remains to show $A^*_L(G)_{\Q}=S^{W_L}_{\Q}/(f-\varphi f \mid f \in S_+^{W_G})_{\Q}$ in the non-special case. 
Using the same argumentation as in the special case we arrive at
$$
A^*_T(G)_{\Q}=S_{\Q}/(f-\varphi f \mid f \in S^{W_G}_+)_{\Q}.
$$
Now by Theorem \ref{ThInv} we know $A^*_L(G)_{\Q}=A^*_T(G)_{\Q}^{W_L}$.
Since $S_{\Q}^{W_L} \hookrightarrow S_{\Q}$ is finite free (\cite[Theorem 2 (d)]{De}) it is also faithfully flat. Hence we obtain
$S_{\Q}^{W_L} \cap IS_{\Q}=IS_{\Q}^{W_L}$ and the assertion follows.
\end{proof}

In the following  we will write $c_i$ for the $i$-th elementary symmetric polynomial in the variables $t_1,\ldots,t_h$ and $c_i^{(j,k)}$ will denote the $i$-th 
elementary symmetric polynomial in the variables $t_j,\ldots,t_k$, where $1 \leq j < k \leq h$ and $1 \leq i \leq k-j+1$.
We then have $\mathbb{Z}[t_1,\ldots,t_n]^{S_{h-d} \times S_d}= \mathbb{Z}[c_1^{(1,h-d)},\ldots,c_{h-d}^{(1,h-d)},c_1^{(h-d+1,h)},\ldots,c_d^{(h-d+1,h)}]$.

\begin{Theorem}
\label{ThSt}
\begin{align*}
A^*(\disp^{h,d}_1) & =A^*_{\GL_{h-d} \times \GL_d}(\GL_h) \\
                   & = \mathbb{Z}[t_1,\ldots,t_n]^{S_{h-d} \times S_d}/((p-1)c_1,\ldots,(p^h-1)c_h),
\end{align*}
where the $c_i^{(1,h-d)}$ resp.\ $c_i^{(h-d+1,h)}$ are the Chern classes of ${}^t\mathcal{L}ie^{\vee}$ resp. $\mathcal{L}ie$.
\end{Theorem}

\begin{proof}
We have that $G_1^{h,d}$ is a split extension of the group $\GL_{h-d} \times \GL_{d}$ by the unipotent
group $\{ \left(
  \begin{array}{ c c }
     E_{h-d} & * \\
     * & E_d
  \end{array} \right) \},$ where $*$ denotes an arbitrary matrix (cf. Remark \ref{RemG_1^{h,d}}). The splitting is given
by the canonical inclusion $\GL_{h-d} \times \GL_d \hookrightarrow G_1^{h,d}$. 
Hence by Lemma \ref{LeUni2} we know 
$$A^*(\disp^{h,d}_1)=A^*_{\GL_{h-d} \times \GL_d}(\GL_h),$$ where the action of $\GL_{h-d} \times \GL_d$ on $\GL_h$ is given
by $\sigma$-conjugation. Since $\GL_{h-d} \times \GL_d$ is special 
with Weyl group $S_{h-d} \times S_d$ we obtain from Proposition \ref{PropChowConj}
$$
A^*_{\GL_{h-d} \times \GL_d}(\GL_h)=\mathbb{Z}[t_1,\ldots,t_n]^{S_{h-d} \times S_d}/((p-1)c_1,\ldots,(p^h-1)c_h).
$$
The assertion that the $c_i^{(1,h-d)}$ resp.\ $c_i^{(h-d+1,h)}$ are the Chern classes of $\mathcal{L}ie$ resp.\ ${}^t\mathcal{L}ie^{\vee}$ follows from
the following simple fact. Let us write $\mathcal{E}_{d}$ resp.\ $\mathcal{E}_{h-d}$ for the vector bundle over $[\ast/\GL_d]$ resp.\ $[\ast/\GL_{h-d}]$ that corresponds to
the canonical representation of $\GL_d$ resp.\ $\GL_{h-d}$. Then $\mathcal{L}ie$ is the pull-back of $\mathcal{E}_d$  under the natural map
$$
\xymatrix{
\disp^{h,d}_1=[GL_h/G_1^{h,d}] \ar[r] & [\ast/(\GL_d \times GL_{h-d})]  \ar[r] & [\ast/GL_d]
}
$$
and similary for ${}^t\mathcal{L}ie^{\vee}$.
\end{proof}

\begin{Corollary}
 $$
\Pic(\disp^{h,d}_1)=\begin{cases} \mathbb{Z}/(p-1)\mathbb{Z} &\mbox{if } d=0,h \\
                        \mathbb{Z} \times \mathbb{Z}/(p-1)\mathbb{Z} &\mbox{else.}
                       \end{cases}
 $$
A generator for the free resp.\ torsion part is $\det(\mathcal{L}ie)$ resp.\ $\det(\mathcal{L}ie \otimes {}^t\mathcal{L}ie^{\vee})$.
\end{Corollary}

\begin{proof}
Note $\Pic \disp_n^{h,d}=A^1 \disp_n^{h,d}$ by \cite[Corollary 1]{EG2}.
\end{proof}

\begin{Remark}
There is also a more direct approach to compute 
the above Picard groups. By using a theorem of Rosenlicht, namely that for irreducible varieties $X$ and $Y$ the
natural map $\mathcal{O}(X)^* \times \mathcal{O}(Y)^* \to \mathcal{O}(X \times Y)^*$ is surjective,
it is not difficult to establish the following exact sequence 
$$
\xymatrix{
\mathcal{O}(X)^*/k^* \ar[r] & \hat{G} \ar[r] & \Pic^G(X) \ar[r] & \Pic(X)
}
$$ 
for $G$ connected and $X$ an irreducible $G$-scheme. The first map assigns to a non-vanishing regular function on $X$ its eigenvalue. 
In our case we have $G=\GL_{h-d} \times \GL_d$ and $X=\GL_h$. Then $\mathcal{O}(\GL_h)^*/k^*=\mathbb{Z}$ with generator
given by the determinant and eigenvalue given by the character $(p-1)(det_{\GL_{h-d}}+\det_{\GL_d}) \in \hat{G}$. Since $\Pic(\GL_h)=0$ we again obtain 
$\Pic^{\GL_{h-d} \times \GL_d}(\GL_h)=\mathbb{Z} \times \mathbb{Z}/(p-1)\mathbb{Z}$.
\end{Remark} 

\begin{Remark}
The fact that $(\det \mathcal{L}ie \otimes \det {}^t\mathcal{L}ie^{\vee})^{p-1}$ is trivial can also be seen directly as follows: 
$(\det \mathcal{L}ie \otimes \det {}^t\mathcal{L}ie^{\vee})^{p-1}$ being trivial means that $\det \mathcal{L}ie \otimes \det {}^t\mathcal{L}ie^{\vee}$ 
is fixed under the pull-back of the Frobenius map $Frob \colon \disp^{2,1}_1 \to \disp^{2,1}_1$ assigning to a display $\mathcal{P}$ 
over an $\mathbb{F}_p$-algebra $R$ the display $\mathcal{P}^{\sigma}$ obtained by base change via the Frobenius $\sigma \colon R \to R$.
But by definition of a truncated display we have an isomorphism $\Psi \colon L \oplus T \cong L^{\sigma} \oplus T^{\sigma}$ of $R$-modules. Taking the determinant
of $\Psi$ yields the desired isomorphism $\det L \otimes \det T \cong \det L^{\sigma} \otimes \det T^{\sigma}$.
\end{Remark}

\begin{Remark}
Let us put this result into context by relating it to the corresponding result for elliptic curves. 
Let $\mathcal{M}_{1,1} \to \Spec k$ denote the moduli stack of elliptic curves. A morphism $\Spec R \to \mathcal{M}_{1,1}$
corresponds to a pair $(C \to \Spec R,\sigma)$ where $C \to \Spec R$ is a smooth projective curve of genus $1$ and $\sigma \colon \Spec R \to C$
is a smooth section. We now have the following diagram
$$
\xymatrix{
\mathcal{M}_{1,1} \ar[r] & BT^{h=2,d=1} \ar[d] \ar[r]^{\phi} & \disp^{h=2,d=1} \ar[d] \\
&  BT_{n=1}^{h=2,d=1} \ar[r]_{\phi_1} & \disp^{h=2,d=1}_{n=1}
}
$$
where $\mathcal{M}_{1,1} \to BT^{h=2,d=1}$ sends an elliptic curve $C$ to its associated Barsotti-Tate group $C[p^{\infty}]$.
Let us consider the pull-back map $A^*(\disp^{2,1}_1) \to A^*(\mathcal{M}_{1,1})$. In characteristic $p$ different from $2$ and $3$ Edidin and Graham computed
$A^*(\mathcal{M}_{1,1})=\mathbb{Z}[t]/(12t)$, where $t$ is given by the first Chern class of the Hodge bundle on $\mathcal{M}_{1,1}$ 
(\cite[Proposition 21]{EG2}). 

By construction of the truncated display functor the pull-back of $\mathcal{L}ie$ to $\mathcal{M}_{1,1}$ is the dual of the Hodge bundle on $\mathcal{M}_{1,1}$.
Since the dual of an elliptic curve is the elliptic curve it follows from Remark \ref{RemVB} that the pull-back of 
$^t \mathcal{L}ie^{\vee}$ is given by the Hodge bundle. Hence $A^*(\disp^{2,1}_1) \to A^*(\mathcal{M}_{1,1})$ is the map
$$
\mathbb{Z}[t_1,t_2]/((p-1)c_1,(p^2-1)c_2) \to \mathbb{Z}[t]/(12t)
$$
that sends $t_1$ to $-t$ and $t_2$ to $t$. Note that $p^2-1$ is divisible by $12$ if and only if $p \geq 5$. In particular,
there can be no such map for $p=2,3$, and we deduce that the description $A^*(\mathcal{M}_{1,1})=\mathbb{Z}[t]/(12t)$ does not hold in characteristic $2$ and $3$.
\end{Remark}

\subsection{The Chow Ring of the Stack of G-Zips} 
Let us first consider the case of F-zips introduced in \cite{MW}. We denote by $\fzip$ the stack of F-zips over a field $k$ of characteristic $p>0$. For $S$ a $k$-scheme $\fzip(S)$ is the groupoid of F-zips over $S$. If $\tau \colon \mathbb{Z} \to \mathbb{Z}_{\geq 0}$ is a function with finite support
we denote by $\fzip^{\tau}$ the open and closed substack of F-zips of type $\tau$. Note that
$$\fzip=\coprod_{\tau}\fzip^{\tau}.$$
The stacks $\fzip^{\tau}$ are smooth Artin algebraic stacks over $k$ which follows for example from the following
representation as a quotient stack.
Let $X_{\tau}$ denote the $k$-scheme whose $S$-valued points are given by
$$X_{\tau}(S)=\{\underline{M}=(M,C^{\bullet},D_{\bullet},\varphi_{\bullet}) \mid \underline{M} \text{ F-zip of type } \tau, M=\mathcal{O}_S^h\}.$$
This is a smooth scheme of dimension $h^2$. Here $h=\sum_{i \in \mathbb{Z}} \tau(i)$ is also called the height of $\underline{M}$.
The group $\GL_h$ acts on $X_{\tau}$ by
$$ G \cdot \underline{M}=(\mathcal{O}_S^h,G(C^{\bullet}),G(D_{\bullet}),G\varphi_{\bullet}(G^{-1})^{\sigma}).$$
It is easy to see that two F-zips over $S$ of the above form are isomorphic if and only if they lie in the same $\GL_h(S)$-orbit. Thus
$$ \fzip^{\tau}=[X_{\tau}/\GL_h]. $$
An F-zip $\underline{M}$ of type $\tau$ with support in $\{0,1\}$ over an $\F_p$-algebra $R$  is just a tuple 
$$\underline{M}=(M,C,D,\varphi_0,\varphi_1),$$
where $M$ is a projective $R$-module with submodules $C$ and $D$, which are direct summands of $M$ and isomorphisms
$$\varphi_0 \colon C^{\sigma} \to M/D, \quad \varphi_1 \colon (M/C)^{\sigma} \to D.$$ 

\begin{Lemma}
\label{LeDispzip}
Let $R$ be an $\mathbb{F}_p$-algebra. Then we have an equivalence of categories
$$ \disp_1(R) \to \coprod_{\tau,\supp(\tau) \in \{0,1\}} \fzip^{\tau}(R) $$
given in the following way
$$ (L,T,\Psi) \mapsto (L \oplus T,T, \Psi^{\sigma}(L^{\sigma}),\Psi^{\sigma} \mid_{T^{\sigma}},\Psi^{\sigma}\mid_{L^{\sigma}}).$$
The above assignment commutes with pulling back. In particular, we get an isomorphism of stacks
$$ \fzip^{\tau} \cong \disp^{\tau(0)+\tau(1),\tau(1)}_1$$
for every type $\tau$ with support lying in $\{0,1\}$.
\end{Lemma}

\begin{proof}
An inverse functor is given by the assignment 
$$ (M,C,D,\varphi_0,\varphi_1) \mapsto (C,M/C,\varphi_0 \oplus \varphi_1).$$
\end{proof}

There is more generally the stack of G-zips introduced in \cite{PWZ}. Here G refers to an arbitrary reductive group.
It is defined as follows. Let $\mathcal{Z}$ be an algebraic zip datum, i.e. a $4$-tupel $(G,P,Q,\varphi)$ consisting of a
split reductive group $G$, parabolic subgroups $P$ and $Q$ and an isogeny $\varphi \colon P/R_u(P) \to Q/R_u(Q)$.
To $\mathcal{Z}$ one associates the group
$$ E_{\mathcal{Z}}=\{(p,q) \in P \times Q \mid \varphi(\pi_P(p))=\pi_Q(q) \}.$$
Now $E_{\mathcal{Z}}$ acts on $G$ by the rule 
$$ ((p,q),g) \mapsto pgq^{-1}$$
and the quotient stack $[G/E_{\mathcal{Z}}]$ is called the stack of $G$-zips. If $G$ is connected
$\mathcal{Z}$ is called a connected zip datum (\cite[Definition 3.1]{PWZ}).

Let us recall how the stack of F-zips is just a special case of this construction. For this let 
$\tau \colon \mathbb{Z} \to \mathbb{Z}_{\geq0}$ be a function with finite support, say $i_1 \leq \ldots \leq i_r$.
If we denote $n_k=\tau(i_k)$, then $(n_1, \ldots, n_r)$ defines a partition of $h=\sum_k n_k$. We denote the standard
parabolic of type $(n_1,\ldots,n_r)$ in $\GL_h$ by $P_{\tau}$. 

\begin{Lemma}
\label{LeG-Zip}
Let $\tau \colon \mathbb{Z} \to \mathbb{Z}_{\geq0}$ be a function with finite support and 
$\mathcal{Z}=(\GL_h,P_{\tau},P_{\tau}^-,\sigma)$ be the algebraic zip datum with $P_{\tau}^-$ the opposite parabolic of $P_{\tau}$ 
and $\sigma$ the Frobenius isogeny. Then there is an isomorphism of stacks
$$ [\GL_h/E_{\mathcal{Z}}] \stackrel{\sim}{\rightarrow} \fzip^{\tau}.$$
\end{Lemma}

\begin{proof}
Let $S$ be an $k$-scheme. We denote by $C^{\bullet}_{\tau}$ the descending filtration 
$$
C^{\bullet}_{\tau}=\mathcal{O}_S^h \supset \mathcal{O}_S^{n_1+\ldots+n_{r-1}} \supset \ldots \supset \mathcal{O}_S^{n_1} \supset 0
$$
in $\mathcal{O}_S^h$ given by the standard flag of type $(n_1,\ldots,n_r)$ and by $D_{\bullet}^{\tau^-}$ the ascending filtration 
$$
D_{\bullet}^{\tau^-}=0 \subset \mathcal{O}_S^{n_r} \subset \ldots \subset \mathcal{O}_S^{n_r+\ldots+n_2} \subset \mathcal{O}_S^h.
$$
given by the flag of type opposite to $(n_1,\ldots,n_r)$. To $g \in \GL_h(S)$ we assign the F-zip 
$$ \underline{M}_g=(\mathcal{O}_S^h,C^{\bullet}_{\tau},g(D_{\bullet}^{\tau^-}),\varphi_{\bullet}), $$
where $\varphi$ is given by the restriction of $g$ to the succesive quotients of $C^{\bullet}_{\tau}$.
Note that we can consider $g$ as a $\sigma$-linear map. 

If $(p,q)$ is an element of $E_\mathcal{Z}$ we get an isomorphim $M_{g} \to M_{pgq^{-1}}$ of F-zips induced by $p$.
The fact that $p$ commutes with the $\varphi_i$ is exactly the condition $\sigma(\pi(p))=\pi(q)$. 
On the other hand if an isomorphism $p \colon M_g \to M_{g'}$ of F-zips is given, we see that $g'^{-1}pg$ preserves
the flag of type opposite to $(n_1,\ldots,n_r)$. Thus $q=g'^{-1}pg \in P_{\tau}^-$ and again the compatibility of $p$ with the
$\varphi_i$ implies the condition $\sigma(\pi(p))=\pi(q)$.
\end{proof}

We can also use Proposition \ref{PropChowConj} to say something about the Chow ring of the stack of $G$-zips for an arbitrary 
connceted algebraic zip datum.

\begin{Definition}
 We call an algebraic zip datum $\mathcal{Z}=(G,P,Q,\varphi)$ special, if $G$ is special.
\end{Definition}

\begin{Theorem}
\label{ThGzip}
Let $\mathcal{Z}=(G,P,Q,\varphi)$ be a connected algebraic zip datum. Let $W_G=W(G,T)$ be the Weyl group of $G$ and $W_L=W(L,T)$ be the Weyl group of a Levi component 
$L$ of $P$ w.r.t. a split maximal torus $T \subset L$ of $G$. Let $g_0 \in G(k)$ such that $\varphi(T)={}^{g_0}T$ and let $\tilde{\varphi} \colon T \to T$ 
denote the composition of $\varphi$ followed by conjugation with $g_0^{-1}$. Then $\tilde{\varphi}$ induces an action on $S=\Sym(\hat{T})$ that 
we will also denote by $\tilde{\varphi}$. We then have
$$A^*([E_{\mathcal{Z}} /G])_{\mathbb{Q}}=S^{W_L}_{\Q}/(f-\tilde{\varphi} f \mid f \in S_+^{W_G})_{\Q}.$$
If $\mathcal{Z}$ is special we have
$$
A^*([E_{\mathcal{Z}} / G])=S^{W_L}/(f-\tilde{\varphi} f \mid f \in S_+^{W_G}).
$$
(Note that the action of $\tilde{\varphi}$ on $S^{W_G}$ is independent of the choice of $g_0$ since two choices differ by an element of $N_G(T)$.)
\end{Theorem}

\begin{proof}
By definition of the group $E_{\mathcal{Z}}$ we have a split exact sequence
$$
\xymatrix{
0 \ar[r] & R_u(P) \times R_u(Q) \ar[r] & E_{\mathcal{Z}} \ar[r] & L \ar[r] & 0,
}
$$
where the splitting is given by $L \hookrightarrow E_{\mathcal{Z}}$, $l \mapsto (l,\varphi(l))$. From Lemma \ref{LeUni2} we deduce 
$$A^*([E_{\mathcal{Z}} / G])_{\Q}=A^*_L(G)_{\Q},$$
where the action of $L$ on $G$ is given by $\varphi$-conjugation. If $G$ is special the above equality holds over $\Z$.
We conclude by Proposition \ref{PropChowConj}.
\end{proof}

\begin{Example}
We consider the case $\mathcal{Z}=(Sp(2n),P,P^{-},\sigma)$, where $\sigma$ denotes the $q$-th power Frobenius.
Recall that $Sp(2n)$ is special and the Weyl group of $Sp(2n)$ is the wreath product 
$S_n \wr (\mathbb{Z}/2\mathbb{Z})=S_n \ltimes (\mathbb{Z}/2\mathbb{Z})^n$. It acts on $\Sym(\hat{T})=\mathbb{Z}[t_1,\ldots,t_n]$ in the following way.
$S_n$ acts by permuting the variables $t_1,\ldots,t_n$ and after identifying $\mathbb{Z}/2\mathbb{Z}=\{\pm 1 \}$ an element 
$(\varepsilon_1,\ldots,\varepsilon_n)\in\mathbb{Z}/2\mathbb{Z}^n$ acts by $(\varepsilon_1,\ldots,\varepsilon_n)\cdot t_i=\varepsilon_i t_i$. 

If $P$ is a Borel we obtain from the above theorem
$$
A^*([E_{\mathcal{Z}} / Sp(2n)]))=\mathbb{Z}[t_1,\ldots,t_n]/((q^2-1)c_1(\underline{t}^2),\ldots,(q^{2n}-1)c_n(\underline{t}^2)).
$$
If $P$ is the maximal parabolic subgroup fixing a maximal isotropic subspace then $L=\GL_n$ and $W_L=S_n$ and therefore
$$
A^*([E_{\mathcal{Z}} / Sp(2n)])=\Z[c_1,\ldots,c_n]/((q^2-1)c_1(\underline{t}^2),\ldots,(q^{2n}-1)c_n(\underline{t}^2)).
$$
\end{Example}

It turns out that a $\Q$-basis of the Chow ring of the stack of G-zips is given by the closures
of the orbits of the action of $E_{\mathcal{Z}}$ on $G$. To prove this let us introduce the naive Chow group of a quotient stack.

\begin{Definition}
 Let $G$ be an algebraic group and $X$ be a $G$-scheme. Let $Z_*([X/G])$ be the free abelian group generated by the set of $G$-invariant subvarieties of 
$X$ graded by dimension. Let $W_i([X/G])$ be the group $\bigoplus_Y k(Y)^G$, where the sum goes over all $G$-invariant subvarieties of $X$ of dimension $i+1$.
There is the usual divisor map $div \colon W_i([X/G]) \to Z_i([X/G])$ and we define the $i$-th naive Chow group of $[X/G]$ to be $$A^o_i[X/G]=Z_i([X/G])/div(W_i([X/G])).$$
\end{Definition}

\begin{Remark}
 There is more generally a definition of naive Chow groups for arbitrary algebraic stacks (\cite[Definition 2.1.4]{Kr}) which
in the case of a quotient stack agrees with the one given above. Thus the above definition is independent of the presentation as a quotient stack. 
\end{Remark}

\begin{Remark}
 There is a natural map $A^o_*[X/G] \to A_*[X/G]$. When $X$ is Deligne-Mumford, i.e. the stabilizer of every point is finite and geometrically reduced, the induced map 
$A^o_*[X/G]_{\mathbb{Q}} \to A_*[X/G]_{\mathbb{Q}}$ is an isomorphism of groups and an isomorphism of rings if $[X/G]$ is smooth (\cite[Theorem 2.1.12 (ii)]{Kr}).
\end{Remark}

The stack of $G$-zips is not Deligne-Mumford. However, we still have the following proposition.

\begin{Proposition}
 Let $G$ be a connected algebraic group and $X$ be an admissible $G$-scheme (cf. Definition \ref{DefAdm}) with finitely many orbits such that the stabilizer 
 of every point is an extension of a finite group by a unipotent group. Then $A^o_*[X/G]_{\mathbb{Q}} \to A_*[X/G]_{\mathbb{Q}}$ is an isomorphism.
\end{Proposition}

\begin{proof}
 We prove this by induction on the number of orbits. Let $U$ denote the open $G$-orbit and $W$ its complement. We have a commutative diagram
$$
\xymatrix{
0 \ar[r] & A^o_*[W/G]_{\mathbb{Q}} \ar[r] \ar[d] & A^o_*[X/G]_{\mathbb{Q}} \ar[d] \ar[r] & A^o_*[U/G]_{\mathbb{Q}} \ar[r] \ar[d] & 0 \\
0 \ar[r] &  A_*[W/G]_{\mathbb{Q}} \ar[r] &  A_*[X/G]_{\mathbb{Q}} \ar[r] &  A_*[U/G]_{\mathbb{Q}} \ar[r] & 0
}
$$
and we claim that the rows of this diagram are exact. Since there are only finitely many orbits every $G$-invariant subvariety $Y$ of $X$ is the closure of a $G$-orbit.
Since $Y$ admits a dense $G$-invariant subset every $G$-invariant rational function on $Y$ is constant. It follows $A^o_*[X/G]=\bigoplus_Z \mathbb{Z}[\bar{Z}]$ where the sum
goes over all $G$-orbits $Z$ of $X$. From this we obtain the exactness of the top row. For the exactness of the lower row we need to see that the pull-back map
$A_*([X/G],1)_{\mathbb{Q}} \to A_*([U/G],1)_{\mathbb{Q}}$ is surjective. But $[U/G]$ is isomorphic to the classifying space of the stablizer group scheme of $U$.
By assumption and Corollary \ref{CorExt2} we get that $A_*([U/G],m)_{\mathbb{Q}} \to A_*(B\{0\},m)_{\mathbb{Q}} $ is an isomorphism. Equivalently the pull-back of the structure
morphism $[U/G] \to \Spec k$ is an isomorphism for the higher Chow groups with rational coefficents and hence the claim follows. 

Now the right vertical arrow is an isomorphism since both groups are isomorphic to $\mathbb{Q}$. By induction we may assume that the first vertical arrow is also an
isomorphism. 
\end{proof}

Recall that an algebraic zip datum $\mathcal{Z}$ is called orbitally finite if $G$ has finitely many $E_{\mathcal{Z}}$-orbits (\cite[Definition 7.2]{PWZ}).

\begin{Theorem}
\label{ThBasis}
 Let $\mathcal{Z}$ be an orbitally finite connected algebraic zip datum and $[G/E_{\mathcal{Z}}]$ be the corresponding stack of $G$-Zips. 
 Then the following assertions hold.
\begin{equivlist}
 \item $A^o_*[G/E_{\mathcal{Z}}]_{\mathbb{Q}} \to A_*[G/E_{\mathcal{Z}}]_{\mathbb{Q}}$ is an isomorphism.
 \item $A^o_*[G/E_{\mathcal{Z}}]=\bigoplus_Z \mathbb{Z} [\bar{Z}]$ where the sum goes over all orbits $Z$.
\end{equivlist}
In particular, the dimension of $A_*[G/E_{\mathcal{Z}}]_{\mathbb{Q}}$ as a $\mathbb{Q}$-vector space is equal to the number of orbits.
\end{Theorem}

\begin{proof}
The assumption of the previous proposition on the stabilizer group schemes hold by \cite[Theorem 8.1]{PWZ}.
\end{proof}

\begin{Corollary}
\label{CorDimG-Zip}
 Let $\mathcal{Z}=(G,P,Q,\varphi)$ be a connected algebraic zip datum and $T$ be a split maximal torus of $G$ in a Levi component $L$ of $P$.
 If $\mathcal{Z}$ is orbitally finite the $\Q$-vectorspace $A^*([E_{\mathcal{Z}} / G])_{\Q}$ is finite dimensional of
 dimension $|W_G/W_L|$, where as usual $W_G=W(G,T)$ is the Weyl group of $G$ and $W_L=W(L,T)$ is the Weyl group of $L$.
\end{Corollary}

\begin{proof}
 By the above theorem $\dim_{\Q} A^*([E_{\mathcal{Z}} / G])_{\Q}$ equals the number of $E_{\mathcal{Z}}$-orbits in $G$.
 This number equals $|W_G/W_L|$ by \cite[Theorem 7.5]{PWZ}.
\end{proof}

In the case of F-zips the above results read as follows.

\begin{Corollary}
\label{CorF-Zip}
 Let $\tau \colon \Z \to \Z_{\geq0}$ be a function with finite support $i_1 \leq \ldots \leq i_r$ and $n_k=\tau(i_k)$. Let $h=\sum_i n_i$ be its height. Then the following
 holds
 \begin{equivlist}
\item $$ A^* \fzip^{\tau} = \mathbb{Z}[t_1,\ldots,t_h]^{S_{n_1} \times \ldots \times S_{n_r}}/((p-1)c_1,\ldots,(p^h-1)c_h)$$
       with  $c_i$ the $i$-th elementary symmetric polynomial in the variables $t_1,\ldots,t_h$.
\item $$ \Pic(\fzip^{\tau})=\mathbb{Z}^{r-1} \times \mathbb{Z}/(p-1)\mathbb{Z}$$
\item $$
      \dim_{\mathbb{Q}} A^*(\fzip^{\tau})_{\mathbb{Q}}= \frac{h!}{n_1! \cdot \ldots \cdot n_r!}
       $$
\end{equivlist}
\end{Corollary}

\subsection{The Chow Ring of $BT_n$}
The goal of this section is to prove the following result.
\begin{Theorem}
 \label{ThChowBT_n}
The pull-back $\phi_n^* \colon A^*(\disp_n) \to A^*(BT_n)$ is injective and an isomorphism after inverting $p$. 
\end{Theorem}

We know that $\disp_n= \coprod_{d\leq h} \disp_n^{h,d}$ is a decomposition into open and closed substacks. The same holds
for $BT_n$ and the morphism $\phi_n$ maps $BT^{h,d}_n$ to $\disp_n^{h,d}$. It suffices to prove
the theorem for the restriction of $\phi_n$ to $BT^{h,d}_n$.
The following proposition is the crucial point in the proof of Theorem \ref{ThChowBT_n}.

\begin{Proposition}
Let $L$ be a field extension of $k$ and $\Spec L \to \disp_n$ be a morphism. Then there is a finite field extension 
$L'$ of $L$ of $p$ power degree and an infinitesimal commutative group scheme $A$ over $L'$ such that the fiber $\phi_n^{-1}(\Spec L')$ is the classifying space of $A$.
\end{Proposition}

\begin{proof}
The diagonal $\Delta \colon BT_n \to BT_n \times_{\disp_n} BT_n$ is flat and surjective by \cite[Theorem 4.7]{La}. This means that two Barsotti-Tate groups
of level $n$ having the same associated display become isomorphic when pulled back to a suitable fppf-covering. It follows that the fiber 
$(BT_n)_L$ of a display $P$ over some field $L$ is a gerbe over $L$. If $L$ is perfect there is a truncated Barsotti-Tate group $G$ over $L$ with $\phi_n(G)=P$, i.e.
$(BT_n)_L$ is a neutral gerbe. In this case $(BT_n)_L=B \underline{\Aut}^o(G)$ where $\underline{\Aut}^o(G)=\Ker(\underline{\Aut}G \to \underline{\Aut}P)$ is
commutative and infinitesimal again by \cite[Theorem 4.7]{La}. If $L$ is not perfect we may consider the perfect hull $L^{p^{-\infty}}$ in an algebraic closure of $L$. 
Then $L \subset L^{p^{-\infty}}$ is purely inseparable and $(BT_n)_L(L^{p^{-\infty}})$ is non-empty. Since $(BT_n)_L(L^{p^{-\infty}})=\varinjlim_{L'}(BT_n)_L(L')$, where
the limit goes over all finite subextensions $L \subset L' \subset L^{p^{-\infty}}$, we find some $L'$ such that $(BT_n)_{L'}$ has a section corresponding to a truncated
Barsotti-Tate group $G$ over $L'$. Thus $A=\underline{\Aut}^o(G)$ and $L'$ have the desired properties.
\end{proof}

\begin{Remark}
 Over the open and closed substack of $BT_n$ consisting of level-n BT-groups of constant dimension $d$ and codimension $c$ the degree of
$\Aut^o(G^{univ})$ is $p^{ncd}$. See Remark 4.8 in \cite{La}.
\end{Remark}

Note that $\disp^{h,d}_n$ and $BT^{h,d}_n$ both admit admissible presentations in the sense of Definition \ref{DefAdm}.
In the case of $\disp^{h,d}_n$ this follows from Theorem \ref{ThRepDisp} and Lemma \ref{LeAdm}. To obtain the assertion for $BT^{h,d}_n$ we use \cite[Proposition 1.8]{We}
which yields a presentation $BT^h_n=[Y^h_n/\GL_{p^{nh}}]$ with $Y^h_n$ quasi-affine and of finite type 
over $k$. Now $BT^h_n$ is smooth over $\Spec k$ (\cite{La}). Hence $Y^h_n$ is also smooth and in particular normal and equidimensional. 

We now consider the flat pull-back map  $$\phi_n^* \colon A_*(\disp_n^{h,d},m) \to A_*(BT^{h,d}_n,m)$$ from Lemma \ref{Lepull-back}.

\begin{Proposition}
 $\phi_n^* \colon A_*(\disp_n^{h,d},m) \to A_*(BT^{h,d}_n,m)$ is an isomorphism after inverting $p$.
\end{Proposition}

\begin{proof}
Let us write $\mathscr{X}=BT^{h,d}_n$ and $\mathscr{Y}=\disp_n^{h,d}$. We fix some $i_o \in \Z$ and show that $\phi_n \colon A_{i_o}(\disp_n^{h,d},m)_p \to A_{i_o}(BT^{h,d}_n,m)_p$ is an isomorphism. 

Consider an approximation of $\mathscr{Y}$ (cf. Convention \ref{ConMixedSpace}) by a quasi-projective scheme $Y \to \mathscr{Y}$ so that
$A_{i_o}(\mathscr{Y},m)=A_{i_o}(Y,m)$ and similary an approximation $X \to \mathscr{X}$ of $\mathscr{X}$. Let $r$ denote the relative dimension 
of $X \to \mathscr{X}$. Let $Z$ be the fibre product $X \times_{\mathscr{Y}} Y$. The morphism $Z \to Y$ is then smooth of relative dimension $r$ and
we need to see that the pull-back $A_{i_o}(Y,m)_p \to A_{i_o+r}(Z,m)_p$ is an isomorphism. 
Note that $Z$ is again quasi-projective since it is open in a vector bundle over the quasi-projective scheme $X$ (cf. Remark \ref{RemAdm}). We have the following cartesian diagram
$$
\xymatrix{
Z_y \ar[d] \ar[r] & \mathscr{X}_{k(y)} \ar[d] \ar[r] & \Spec k(y) \ar[d] \\
Z \ar[r] \ar[d] & \mathscr{X}_Y \ar[d] \ar[r] & Y \ar[d] \\
X \ar[r] & \mathscr{X} \ar[r] & \mathscr{Y}
}
$$
By Lemma \ref{LeFiber} it suffices to see that $A_i(\Spec k(y),m)_p \to A_{i+r}(Z_y,m)_p$ with $i=i_o-\dim \bar{ \{y\} }$ is an isomorphism. 
According to the previous proposition there is a finite field extension $K$ of $k(y)$ of $p$-power degree such that $\mathscr{X}_K=BA$ holds for an infinitesimal
group scheme $A$ over $K$. 

Since $Z_K$  is open in a vector bundle over $\mathscr{X}_K$ of rank $r$ we have $Z_K=U/A$, where $U$ is open in a representation 
$V$ of $A$. Note that $V$ is of dimension $r$. 
Hence by chosing $\codim X^c$ to be big enough, we may assume $A_i(\Spec K,m) \to A_{i+r}(U,m)$ is an isomorphism. Since $A$ is of $p$-power
degree it follows that the map $A_i(\Spec K,m)_p \to A_{i+r}(Z_K,m)_p$ is an isomorphism . Now since the field extension $K \supset k(y)$ is of 
$p$-power degree it follows from Lemma \ref{LeCov} that $A_i(\Spec k(y),m)_p \to A_{i+r}(Z_y,m)_p$ is also an isomorphism. We are done.
\end{proof}

\begin{proof}(of Theorem \ref{ThChowBT_n})
Since $BT_n$ and $\disp_n$ are smooth the pull-back $(\phi_n)_p^* \colon A^*(\disp_n)_p \to A^*(BT_n)_p$ is an isomorphism by Lemma \ref{Lepull-back} and the proposition above. We already know 
$A^*(\disp_n)$ is $p$-torsion free by Theorem \ref{ThSt} and Theorem \ref{Thhigherdisp}. Thus $\phi_n^*$ is injective. 
\end{proof}

 Gathering the results of Chapter 4 we obtain
 
\begin{Theorem}
\begin{equivlist}
\item We have 
$$A^*(BT^{h,d}_n)_{p}=\mathbb{Z}[p^{-1}][t_1,\ldots,t_h]^{S_d \times S_{h-d}}/((p-1)c_1,\ldots,(p^h -1)c_h),$$ 
where $c_i$ denotes the $i$-th elementary symmetric polynomial in the variables
 $t_1,\ldots,t_h$ and $t_1,\ldots,t_d$ resp.\ $t_{d+1},\ldots,t_h$ are the Chern roots of $\mathcal{L}ie$ resp.\ ${}^t\mathcal{L}ie^{\vee}$.
\item $\dim_{\mathbb{Q}}A^*(BT^{h,d}_n)_{\mathbb{Q}}=\binom{h}{d}$ and a basis is given by the cyclces of the closures of the $EO$-Strata.
\item $$  (\Pic BT^{h,d}_n)_p= \begin{cases} \mathbb{Z}[p^{-1}]/(p-1) & \mbox{if } d=0,h \\
                                          \mathbb{Z}[p^{-1}] \times \mathbb{Z}[p^{-1}]/(p-1) &\mbox{else,}
                            \end{cases}$$
where the generator for the free resp.\ torsion part is $\det(\mathcal{L}ie)$ resp.\ $\det(\mathcal{L}ie \otimes {}^t\mathcal{L}ie^{\vee})$.
\end{equivlist}
\end{Theorem}

\begin{proof}
 By Theorem \ref{ThChowBT_n} we know $A^*(\disp_n^{h,d})_p \cong A^*(BT^{h,d}_n)_p$. Morever, we have $A^*(\disp_n^{h,d}) \cong A^*(\disp^{h,d}_1)$ by Theorem \ref{Thhigherdisp} and
 $A^*(\disp^{h,d}_1)$ was computed in Theorem \ref{ThSt}. This proves part (i).
 By Lemma \ref{LeDispzip} and Lemma \ref{LeG-Zip} we know that $\disp^{h,d}_1$ is isomorphic to the stack $[\GL_h/E_{\mathcal{Z}}]$ corresponding to the Frobenius zip datum
 $\mathcal{Z}=(\GL_h,P,P^-,\sigma)$, where $P$ is the standard parabolic of type $(d,h)$, $P^-$ is the opposite parabolic and $\sigma$ is the Frobenius isogeny.
 Now the dimension of $A^*(\disp^{h,d}_1)_{\Q}$ as a $\Q$-vectorspace follows from Corollary \ref{CorF-Zip} and a basis is given by Theorem \ref{ThBasis}. This proves (ii).
 Finally (iii) follows from (i) together with the fact that $\Pic BT^{h,d}_n=A^1(BT^{h,d}_n)$.
\end{proof}

University of Paderborn, D-33098 Paderborn 

Dennis.Brokemper@math.uni-paderborn.de 

\end{document}